\documentclass[a4paper]{amsart}
\usepackage[normalem]{ulem}
\usepackage{amsmath,amssymb,amsthm,amsfonts}
\usepackage[foot]{amsaddr}
\usepackage[utf8]{inputenc}
\usepackage{xspace}
\usepackage{standalone}
\usepackage{tikz}

\newtheorem{theorem}{Theorem}[section]
\newtheorem*{result}{Theorem}
\newtheorem{lemma}[theorem]{Lemma}
\newtheorem{proposition}[theorem]{Proposition}
\newtheorem{corollary}[theorem]{Corollary}
\theoremstyle{definition} \newtheorem{definition}[theorem]{Definition}

\theoremstyle{remark} \newtheorem{remark}[theorem]{Remark}
\numberwithin{equation}{section}

\usepackage{amssymb}
\usepackage{newunicodechar}
\newunicodechar{…}{\ldots}
\newunicodechar{ }{~} % non breakable space

\newunicodechar{α}{\alpha}% \newunicodechar{Α}{\Alpha}
\newunicodechar{β}{\beta}% \newunicodechar{Β}{\Beta}
\newunicodechar{γ}{\gamma} \newunicodechar{Γ}{\Gamma}
\newunicodechar{δ}{\delta} \newunicodechar{Δ}{\Delta}
\newunicodechar{ε}{\epsilon}% \newunicodechar{Ε}{\Epsilon}
\newunicodechar{ζ}{\zeta}% \newunicodechar{Ζ}{\Zeta}
\newunicodechar{η}{\eta}% \newunicodechar{Η}{\Eta}
\newunicodechar{θ}{\theta} \newunicodechar{Θ}{\Theta}
\newunicodechar{ι}{\iota}% \newunicodechar{Ι}{\Iota}
\newunicodechar{κ}{\kappa}% \newunicodechar{Κ}{\kappa}
\newunicodechar{λ}{\lambda} \newunicodechar{Λ}{\Lambda}
\newunicodechar{μ}{\mu}% \newunicodechar{μ}{\mu}
\newunicodechar{ν}{\nu}% \newunicodechar{ν}{\nu}
\newunicodechar{ξ}{\xi} \newunicodechar{Ξ}{\Xi}
\newunicodechar{π}{\pi} \newunicodechar{Π}{\Pi}
\newunicodechar{ρ}{\rho}% \newunicodechar{Ρ}{\Rho}
\newunicodechar{σ}{\sigma} \newunicodechar{Σ}{\Sigma}
\newunicodechar{τ}{\tau}% \newunicodechar{Τ}{\Tau}
\newunicodechar{υ}{\upsilon} \newunicodechar{Υ}{\Upsilon}
\newunicodechar{φ}{\varphi} \newunicodechar{Φ}{\Phi}
\newunicodechar{χ}{\chi}% \newunicodechar{Χ}{\Chi}
\newunicodechar{ψ}{\psi} \newunicodechar{Ψ}{\Psi}
\newunicodechar{ω}{\omega} \newunicodechar{Ω}{\Omega}

\newunicodechar{≤}{\leq}
\newunicodechar{≥}{\geq}
\newunicodechar{≠}{\neq}
\newunicodechar{±}{\pm}
\newunicodechar{¬}{\neg}
\newunicodechar{×}{\times}
\newunicodechar{≃}{\simeq}
\newunicodechar{≲}{\lesssim}
\newunicodechar{≳}{\grtsim}

\newunicodechar{₁}{_1} \newunicodechar{¹}{^1}
\newunicodechar{₂}{_2} \newunicodechar{²}{^2}
\newunicodechar{₃}{_3} \newunicodechar{³}{^3}
\newunicodechar{₄}{_4} \newunicodechar{⁴}{^4}
\newunicodechar{₅}{_5} \newunicodechar{⁵}{^5}
\newunicodechar{₆}{_6} \newunicodechar{⁶}{^6}
\newunicodechar{₇}{_7} \newunicodechar{⁷}{^7}
\newunicodechar{₈}{_8} \newunicodechar{⁸}{^8}
\newunicodechar{₉}{_9} \newunicodechar{⁹}{^9}
\newunicodechar{₀}{_0} \newunicodechar{⁰}{^0}

\newcommand{\cardone}{κ}
\newcommand{\cardtwo}{λ}
\newcommand{\cardthree}{μ}

\newcommand{\ordone}{α}
\newcommand{\ordtwo}{β}
\newcommand{\ordthree}{δ}
\newcommand{\ordfour}{γ}

\newcommand{\cardinal}[1]{\left\vert #1 \right\vert}

\newcommand{\card}[1]{\cardinal{#1}}
\newcommand{\order}[1]{ω_{#1}}

\newcommand{\cofinality}[1]{\operatorname{cf}(#1)}
\newcommand{\diag}[1]{\operatorname{diag}\!\left( #1 \right)}

\newcommand{\thelatticeweak}[1]{Π(#1)}
\newcommand{\thelattice}[1]{Π_{#1}}
\newcommand{\equ}[1]{\operatorname{Equ}(#1)}
\newcommand{\compl}[1]{\operatorname{compl}(#1)}

\newcommand{\lub}{\lor}
\newcommand{\glb}{\land}
\DeclareRobustCommand{\stirling}{\genfrac\{\}{0pt}{}}

\newcommand{\setof}[2]{\left\{\,#1\ \mid\ #2\,\right\}}
\newcommand{\sequenceof}[2]{\left(#1 \right)_{#2}}
\newcommand{\setsub}{\setminus}

\newcommand{\supth}{^{\text{th}}}
\usepackage{mathrsfs}
\newcommand{\powerset}[1]{\mathscr{P}\!\left(#1\right)}

\newcommand{\chain}[1]{\mathbf{#1}}
\newcommand{\poset}[1]{\mathbf{#1}}
\newcommand{\parti}[1]{\mathcal{#1}}
\newcommand{\keyframe}[1]{\parti{K}_{#1}}
\newcommand{\interm}[2]{\parti{K}_{#1}^{#2}}
\newcommand{\intermediate}[1]{\interm{\ordthree}{#1}}

\newcommand{\RESTRICT}[2]{#1 \sqcap #2} % intersection like notation
\newcommand{\restrict}[2][\cardtwo]{\RESTRICT{#2}{#1}}
\newcommand{\blockrestrict}[2]{#1\cap #2}

\newcommand{\bigeqinpart}[3]{\ensuremath{#1
     \underset{{#2}}{\equiv} #3}\xspace}

\let\eqinpart\smalleqinpart
\let\noteqinpart\smallnoteqinpart

\newcommand{\lowchain}[2]{#1^-_{#2}}
\newcommand{\uppchain}[2]{#1^+_{#2}}
\newcommand{\lowerchain}[2]{\lowchain{\chain{#1}}{#2}}
\newcommand{\upperchain}[2]{\uppchain{\chain{#1}}{#2}}

\newcommand{\Cplus}{\glb \upperchain{C}{x}}
\newcommand{\Cminus}{\lub \lowerchain{C}{x}}

\usepackage{ifthen}
\usepackage{calc}

\usetikzlibrary{positioning}
\usetikzlibrary{patterns}

\tikzset{
        hatch distance/.store in=\hatchdistance,
        hatch distance=10pt,
        hatch thickness/.store in=\hatchthickness,
        hatch thickness=2pt
    }

    \makeatletter
    \pgfdeclarepatternformonly[\hatchdistance,\hatchthickness]{flexible hatch}
    {\pgfqpoint{0pt}{0pt}}
    {\pgfqpoint{\hatchdistance}{\hatchdistance}}
    {\pgfpoint{\hatchdistance-1pt}{\hatchdistance-1pt}}%
    {
        \pgfsetcolor{\tikz@pattern@color}
        \pgfsetlinewidth{\hatchthickness}
        \pgfpathmoveto{\pgfqpoint{0pt}{0pt}}
        \pgfpathlineto{\pgfqpoint{\hatchdistance}{\hatchdistance}}
        \pgfusepath{stroke}
    }
\makeatother

\pgfdeclarelayer{background}
\pgfdeclarelayer{foreground}
\pgfsetlayers{background,main,foreground}

\tikzstyle{common}=[rectangle, draw, align=left, text
height=height("0"), rounded corners=2pt]
\tikzstyle{keyframe}=[common, fill=black!20]
\tikzstyle{diagonal}=[common, fill=white,
                      pattern=flexible hatch,
                      hatch distance=10pt,
                      hatch thickness=1pt,
                      pattern color = black!30]
\tikzstyle{shaded}=[common, color=black!50, fill=white]
\tikzstyle{fake}=[common, color=white, fill=white]
\tikzstyle{sparse}=[node distance=2cm]
\tikzstyle{vdots}=[align=center, sparse]

\tikzstyle{arrow}=[color=black!20, ->, shorten >=2pt, shorten <=2pt]

\tikzstyle{textshaded}=[rectangle, align=left, color=black!50,
fill=white]

\tikzset{
  other/.style={},
  textother/.style={},
  diag/.style={},
  leave 1/.style={}
  leave 2/.style={}
}

\newcommand{\word}{\ldots}

\makeatletter
\newcommand{\@side}[1]{\ifnum#1=0 west\else east\fi}

\newcounter{@reccalls} % number of recursive call
\newcounter{@totalcol} % number of columns ~ 2^@reccalls
\newcounter{@mycolumn} % current column.
\newcounter{@bottom} % number of bottom nodes (\vdots)

\newlength{\@reclength}
\newlength{\@oldlength}
\newlength{\@treewidth}
\newcounter{@treedepth}

\def\@toto|#1|{\@totoiter#1\bogo\bogototo\bogobogo}
\def\@totoiter#1#2\bogototo#3\bogobogo{%
  \stepcounter{@reccalls}%
  \setlength{\@oldlength}{\@reclength}
  \setlength{\@reclength}{\@oldlength * \real{0.48}}
  \ifx#1f\else
  \setcounter{@mycolumn}{\value{@mycolumn}*2+\the\numexpr#1\relax}
  \fi
  \ifx\bogo#2%
  \gdef\@myside{\@side{#1}} \gdef\@parent{#3}
  \setcounter{@totalcol}{1}
  \else% recursive call
  \@totoiter#2\bogototo#3#1\bogobogo%
  \setcounter{@totalcol}{\value{@totalcol}*2}
  \fi}

\newcommand{\@mylength}[1]{%
  \setcounter{@reccalls}{0}
  \setcounter{@totalcol}{0}
  \setcounter{@mycolumn}{0}
  \@toto|#1|
  \stepcounter{@mycolumn}
}

\def\@remfirst#1#2+{#2}
\newcommand{\remfirst}[1]{\@remfirst#1+}

\newcommand{\@drawintermediate}[4]{%
  \ifnum\value{@mycolumn}=#3% diagonal
    \node[minimum width=\@reclength, fake] (west#1zfake)
    [below=of #1.west, anchor=west] {$#20\word$};
    \node[minimum width=\@reclength, fake] (east#1zfake)
    [below=of #1.east, anchor=east] {$#21\word$};
    \node[minimum width=\@reclength, diag] (west#1z)
    [below=of #1.west, anchor=west] {$#20\word$};
    \node[minimum width=\@reclength, diag] (east#1z)
    [below=of #1.east, anchor=east] {$#21\word$};
  \else% not diagonal
  \node[minimum width=#4, other] (#1z)
  [below=of #1.center, anchor=center] {$#2\word$};
  \fi
}

\newcommand{\@drawbelowaux}[5]{%
  \@drawintermediate{#1}{#2}{#3}{#4}
  \ifnum\value{@totalcol}=#3% end: rename keyframe(s)
  \ifnum\value{@mycolumn}=#3% diagonal: split
  \node[minimum width=\@reclength, keyframe] (\@uppername0)
  [below=of #1.west, anchor=west] {$\@upperlabel0\word$};
  \node[minimum width=\@reclength, keyframe] (\@uppername1)
  [below=of #1.east, anchor=east] {$\@upperlabel1\word$};
  \else
  \node[minimum width=\@reclength, keyframe] (\@uppername#5)
  [below=of #1.center, anchor=center] {$\@upperlabel#5\word$};
  \fi
  \else %
  \ifnum\value{@mycolumn}=#3% diagonal: split
  \@drawbelowaux{west#1z}{#20}{\the\numexpr#3+1\relax}{\@reclength}{0}
  \@drawbelowaux{east#1z}{#21}{\the\numexpr#3+1\relax}{\@reclength}{1}
  \else
  \@drawbelowaux{#1z}{#2}{\the\numexpr#3+1\relax}{#4}{#5}
  \fi
  \fi
}

\newcommand{\@drawbelow}[2]{%
  \gdef\@uppername{#1}
  \gdef\@upperlabel{#2}
  \@drawbelowaux{#1}{#2}{1}{\@oldlength}{}
}

\newcommand{\@drawtree}[1]{%
  \setlength{\@reclength}{\@treewidth}
  \@mylength{#1}
  \ifthenelse{\value{@reccalls}=1}{% root of the tree
    \node[minimum width=\@oldlength, keyframe] (#1) {$\word$};
    \@drawbelow{#1}{}
  }{% not root
    \@drawbelow{#1}{\remfirst{#1}}
  }

  \begin{pgfonlayer}{\@arrowlayer}
    \draw[arrow] (#1.south) -- (#10.north);
    \draw[arrow] (#1.south) -- (#11.north);
  \end{pgfonlayer}

  \ifthenelse{\value{@reccalls}>\value{@treedepth}}{%
    \stepcounter{@bottom}
    \node[vdots, minimum width=\@reclength] (v\the@bottom)
    [below=of #10.center, anchor=center] {$\vdots$};
    \stepcounter{@bottom}
    \node[vdots, minimum width=\@reclength] (v\the@bottom)
    [below=of #11.center, anchor=center] {$\vdots$};
  }{%
    \@drawtree{#10}
    \@drawtree{#11}
  }
}

\newcounter{@ones} \newcounter{@expoones}
\def\@totoones|#1|{\@totoonesiter#1\bogo\bogototo}
\def\@totoonesiter#1#2\bogototo{%
  \ifx#11
  \stepcounter{@ones}
  \setcounter{@expoones}{\value{@expoones}*2}
  \else
  \relax
  \fi
  \ifx\bogo#2%
  \relax
  \else% recursive call
  \@totoonesiter#2\bogototo%
  \fi}

\newcommand{\@countingones}[1]{%
  \setcounter{@ones}{0}
  \setcounter{@expoones}{1}
  \@totoones|#1|
}

\newcommand{\@numberlevelsaux}[2]{%
  \ifnum#2<\value{@expoones}
  \node[textother] (n#1) [right=of #1] {$\mathcal{K}_{\the@ones}^{#2}$};
  \@numberlevelsaux{#1z}{\the\numexpr#2+1\relax}
  \else
  \relax
  \fi
}

\newcommand{\@numberlevels}[2]{%
  \@countingones{#2}
  \node (n#1) [right=of #2] {$\mathcal{K}_{#1}$};
  \ifnum#1>\value{@treedepth}
  \relax
  \else
  \@numberlevelsaux{#2z}{1}
  \@numberlevels{\the\numexpr#1+1\relax}{#21}
  \fi
}

\newcommand{\@leaves}[1]{%
  \tikzset{leave 1/.style={keyframe, sparse}}
  \def\@leavelabelw{$\bullet$}
  \tikzset{leave 2/.style={keyframe, sparse}}
  \def\@leavelabele{$\bullet$}
  \ifnum#1>\value{@bottom}
  \node (nomega) [right=of l\the\numexpr#1-1\relax east]
  {$\mathcal{K}_{\omega}$};
  \else
  \ifnum#1<\value{@bottom}
  \tikzset{leave 2/.style={text height=height("0"), sparse}}
  \def\@leavelabele{\ldots}
  \fi
  \ifnum#1>1
  \tikzset{leave 1/.style={text height=height("0"), sparse}}
  \def\@leavelabelw{\ldots}
  \fi
  \node[leave 1, minimum width=\@reclength] (l#1west)
  [below=of v#1.west, anchor=west] {\@leavelabelw};
  \node[leave 2, minimum width=\@reclength] (l#1east)
  [below=of v#1.east, anchor=east] {\@leavelabele};
  \@leaves{\the\numexpr#1+1\relax}
  \fi
}

\newcommand{\@drawtreeaux}[2]{%
  \setlength{\@treewidth}{#2}
  \setcounter{@treedepth}{#1}
  \setcounter{@bottom}{0}
  \@drawtree{f}
  \@numberlevels{0}{f}
  \setlength{\@oldlength}{\@reclength}
  \setlength{\@reclength}{\@oldlength * \real{0.48}}
  \@leaves{1}
}

\newcommand{\fulltree}[2]{%
  \gdef\@arrowlayer{background}
  \tikzset{
    other/.style=shaded,
    textother/.style=textshaded,
    diag/.style=diagonal
  }%
  \@drawtreeaux{#1}{#2}
}

\newcommand{\keyframes}[2]{%
  \gdef\@arrowlayer{foreground}
  \tikzset{
    other/.style=fake,
    textother/.style=fake,
    diag/.style=fake
  }%
  \@drawtreeaux{#1}{#2}
}

\usepackage{ifthen}
\usepackage{calc}

\usetikzlibrary{positioning}
\usetikzlibrary{patterns}

\tikzset{
        hatch distance/.store in=\hatchdistance,
        hatch distance=10pt,
        hatch thickness/.store in=\hatchthickness,
        hatch thickness=2pt
    }

    \makeatletter
    \pgfdeclarepatternformonly[\hatchdistance,\hatchthickness]{myhatch}
    {\pgfqpoint{0pt}{0pt}}
    {\pgfqpoint{\hatchdistance}{\hatchdistance}}
    {\pgfpoint{\hatchdistance-1pt}{\hatchdistance-1pt}}%
    {
        \pgfsetcolor{\tikz@pattern@color}
        \pgfsetlinewidth{\hatchthickness}
        \pgfpathmoveto{\pgfqpoint{0pt}{0pt}}
        \pgfpathlineto{\pgfqpoint{\hatchdistance}{\hatchdistance}}
        \pgfusepath{stroke}
    }
\makeatother

\tikzstyle{lambda}=[rectangle, draw, align=left, text
height=height("0")+1pt, rounded corners=2pt]

\tikzstyle{hatched}=[fill=white,
                     pattern=myhatch,
                     hatch distance=8pt,
                     hatch thickness=1pt,
                     pattern color = black!30]

\tikzstyle{kappa}=[rectangle, draw, dashed, align=left, text
height=height("0"), rounded corners=5pt]

\tikzstyle{label}=[text=black]

\tikzstyle{elem}=[circle,fill=black,inner sep=1pt]

\tikzstyle{arrow}=[color=black, ->, shorten >=2pt, shorten <=2pt]

\begin{document}

\title{Chains, Antichains, and Complements in Infinite Partition
  Lattices}

% Remove any unused author tags.

% author one information
\author{James Emil Avery}
\address{
  James Avery:
  Niels Bohr Institute, University of Copenhagen\\
  Blegdamsvej 17, 2100 Copenhagen \O\\
  Denmark\\
  E-Mail: \texttt{avery@diku.dk}
}  \curraddr{}
%\email{avery@diku.dk}
 \thanks{}

% author two information
\author{Jean-Yves Moyen}
\address{
  Jean-Yves Moyen:
  Laboratoire d'Informatique de Paris Nord, Universit{\' e} Paris XIII\\
  99, avenue J.-B. Cl{\' e}ment\\
  93430 Villetaneuse\\
  France\\
  E-Mail: \texttt{Jean-Yves.Moyen@lipn.univ-paris13.fr}
} 
\curraddr{}
%\email{Jean-Yves.Moyen@lipn.univ-paris13.fr}
\thanks{}

% author three information
\author[P.~Ruzicka]{Pavel Ruzicka}
\address{
  Pavel Ruzicka: 
  Department of Algebra , Room 307\\
  Sokolovska 83\\
  186 75, Prague\\
  Czech Republic\\
  E-Mail: \texttt{ruzicka@karlin.mff.cuni.cz}
}
\curraddr{}
\thanks{}

% author four information
\author{Jakob Grue Simonsen} 
\address{
  Jakob Grue Simonsen:
  Department of Computer Science, University of Copenhagen
  (DIKU)\\
  Njalsgade 128-132, 2300 Copenhagen S\\
  Denmark\\
  E-Mail: \texttt{simonsen@diku.dk}
} 
\curraddr{} 
\thanks{}

% \author[J.~E.~Avery]{James Emil Avery}
% \email{avery@diku.dk}
% \address{Niels Bohr Institute, University of Copenhagen\\
%   Blegdamsvej 17, 2100 Copenhagen \O\\
%   Denmark} 
% \curraddr{}
% \thanks{}

% % author two information
% \author[J.-Y.~Moyen]{Jean-Yves Moyen}
% \email{Jean-Yves.Moyen@lipn.univ-paris13.fr}
% \address{Laboratoire d'Informatique de Paris Nord, Universit{\' e} Paris XIII\\
%   99, avenue J.-B. Cl{\' e}ment\\
%   93430 Villetaneuse\\
%   France}
% \curraddr{Department of Computer Science, University of Copenhagen
%   (DIKU)\\
%   Njalsgade 128-132, 2300 Copenhagen S\\
%   Denmark}
% \thanks{}

% % author four information
% \author[J.~G.~Simonsen]{Jakob Grue Simonsen} 
% \email{simonsen@diku.dk}
% \address{Department of Computer Science, University of Copenhagen
%   (DIKU)\\
%   Njalsgade 128-132, 2300 Copenhagen S\\
%   Denmark
% } 
% \curraddr{} 
% \thanks{}

\subjclass[2020]{Primary 06B05; 06C15}

\keywords{Partition lattice; order theory; antichain; chain; complement;
  orthocomplemented lattice; cardinal; ordinal}

\date{}

\dedicatory{}

% colorblind friendly color scheme
% http://jfly.iam.u-tokyo.ac.jp/color/#pallet
\definecolor{CBorange}{RGB}{230,159,0}
\definecolor{CBskyblue}{RGB}{86,180,230} % light blue
\definecolor{CBgreen}{RGB}{0,158,115} % green
\definecolor{CByellow}{RGB}{240,228,66}
\definecolor{CBblue}{RGB}{0,114,178}
\definecolor{CBred}{RGB}{213,94,0} % vermillon
\definecolor{CBpink}{RGB}{204,121,167}

\begin{abstract}
  \vspace{-.3cm} We consider the partition lattice
  $\thelattice{\cardone}$ on any set of transfinite cardinality
  $\cardone$ and properties of $\thelattice{\cardone}$ whose analogues
  do not hold for finite cardinalities. Assuming the Axiom of Choice
  we prove: %
  (I) the cardinality of any maximal well-ordered chain is always
  exactly $\cardone$; %
  (II) there are maximal chains in $\thelattice{\cardone}$ of
  cardinality $> \cardone$; %
  (III) if, for every cardinal $\cardtwo < \cardone$, we have
  $2^{\cardtwo} < 2^\cardone$, there exists a maximal chain of
  cardinality $< 2^{\cardone}$ (but $≥ \cardone$) in
  $\thelattice{2^\cardone}$; %
  (IV) every non-trivial maximal antichain in $\thelattice{\cardone}$
  has cardinality between $\cardone$ and $2^{\cardone}$, and these
  bounds are realized. Moreover we can construct maximal antichains of
  cardinality $\max(\cardone, 2^{\cardtwo})$ for any
  $\cardtwo ≤ \cardone$; %
  (V) all cardinals of the form $\cardone^\cardtwo$ with
  $0 ≤ \cardtwo ≤ \cardone$ occur as the number of complements to
  some partition $\parti{P} \in \thelattice{\cardone}$, and only these
  cardinalities appear. Moreover, we give a direct formula for the
  number of complements to a given partition; %
  (VI) Under the Generalized Continuum Hypothesis, the cardinalities
  of maximal chains, maximal antichains, and numbers of complements
  are fully determined, and we provide a complete characterization.
\end{abstract}

\maketitle

Let $\cardone$ be a cardinal and let $S$ be a set of cardinality
$\cardone$. The set of partitions of $S$ forms a lattice when endowed
with the binary relation $≤$, called \emph{refinement}, defined by
$\parti{P} ≤ \parti{Q}$ if and only if each block of $\parti{P}$ is
a subset of a block of $\parti{Q}$.  This lattice is called the
\emph{partition lattice} on $S$, and is denoted $\thelatticeweak{S}$.
By the standard correspondence between partitions and equivalence
relations, it follows that $\thelatticeweak{S}$ is isomorphic
to the lattice $\equ{S}$ of equivalence relations on $S$ ordered by
set inclusion on $S \times S$.

As the particulars of $S$ do not affect the order-theoretic properties
of $\thelatticeweak{S}$ we shall without loss of generality restrict
our attention to the lattice $\thelattice{\cardone} =
\thelatticeweak{\cardone}$. Initiated by a seminal paper by Ore
\cite{Ore42}, many of the properties of $\thelattice{\cardone}$ that
hold for arbitrary cardinals $\cardone$ are well-known.  Indeed, it is
known that $\thelattice{\cardone}$ is complete, matroid (hence
atomistic and semimodular), non-modular (hence non-distributive) for
$\cardone ≥ 4$, relatively complemented (hence complemented), and
simple \cite[\S 8-9]{Birkhoff:lattice}, \cite{RivalStanford:algpart},
\cite[Sec.\ IV.4]{Gratzer:genlattice}.

For properties depending on $\cardone$, only a few results exist in
the literature for infinite $\cardone$. Cz{\'e}dli has proved that if
there is no inaccessible cardinal $≤ \cardone$ then the following
holds: If $\cardone ≥ 4$, $\thelattice{\cardone}$ is generated by
four elements \cite{Czedli:four}, and if $\cardone ≥ 7$,
$\thelattice{\cardone}$ is (1+1+2)-generated \cite{Czedli:oneonetwo}
(for $\cardone = \aleph_0$, slightly stronger results hold
\cite{Czedli:countable}). It appears that no further results are
known, beyond those holding for all cardinalities, finite or
infinite. The aim of the present work is to prove a number of results
concerning $\thelattice{\cardone}$ that depend on $\cardone$ being an
infinite cardinal.

\section{Preliminaries and notation}\label{sec:prels}
% Notations
%% cardinals, ordinals.

We work in ZF with the Axiom of Choice (AC). As usual, a set $S$ is
\emph{well-ordered} if and only if it is totally ordered and every
non-empty subset of $S$ has a least element. Throughout the paper, we
use von Neumann's characterization of ordinals: a set $S$ is an
ordinal if and only if it is strictly well-ordered by $\subsetneq$ and
every element of $S$ is a subset of $S$. The \emph{order type} of a
well-ordered set $S$ is the (necessarily unique) ordinal $\ordone$
that is order-isomorphic to $S$.  Cardinals and ordinals are denoted
by Greek letters $\ordone, \ordtwo, \ordfour, \ordthree, \ldots$ for
ordinals and $\cardone, \cardtwo, \ldots$ for cardinals. We denote by
$\order{\cardone}$ the initial ordinal of $\cardone$, and by
$\card{\ordone}$ the cardinality of $\ordone$. The cardinality of a
set is denoted $\cardinal{S}$ and its powerset is denoted
$\powerset{S}$. For a cardinal $\cardone$, we denote by $\cardone^+$
its successor and by $\cardone^-$ its predecessor cardinal. Note that
$\cardone^-$ is defined only if $\cardone$ is a successor cardinal.

Many standard results on cardinal arithmetic can be found in~\cite{HolzSteffensWeitz}, among
other places, and are used frequently throughout the proofs.

\smallskip

Recall that a \emph{chain} in a poset $(\poset{P},≤)$ is a subset of
$\poset{P}$ that is totally ordered by $≤$. Similarly, an
\emph{antichain} in $(\chain{P},≤)$ is a subset of $\chain{P}$ such
that any two distinct elements of the subset are $≤$-incomparable. A
chain (respectively, antichain) in $(\chain{P},≤)$ is \emph{maximal}
if no element of $\chain{P}$ can be added to the chain without losing
the property of being a chain (respectively, antichain). Observe that
if $\poset{P}$ contains a bottom element, $\bot$, or a top element, $\top$, it
belongs to any maximal chain. A chain $\chain{C}$ in $(\chain{P},≤)$ is
\emph{saturated} if, for any two elements $\parti{Q}<\parti{S}$ of the
chain, there is no element $\parti{R} \in \chain{P} \setsub \chain{C}$
such that $\parti{Q} < \parti{R} < \parti{S}$ and
$\chain{C} \cup \{\parti{R}\}$ is a chain; notably, a chain containing
$\bot$ and $\top$ is maximal if and only if it is saturated. We say
that a chain is \emph{endpoint-including} if it contains a least and a
greatest element, not necessarily equal to $\bot$ and
$\top$, respectively. By the Maximal Chain Theorem
\cite{Hausdorff:grund}, every chain in a poset is contained in a
maximal chain.

\smallskip

%% partitions, set of partitions.
We denote partitions (and equivalences) of $\cardone$ by capital
italic Roman letters $\parti{P},\parti{Q},\ldots$, and denote subsets
of $\thelattice{\cardone}$ such as chains and antichains by capital
boldface letters $\chain{C}, \chain{D},\ldots$;
%% blocks
If $\parti{P} = \{B_\ordthree\}$ is a partition, we call its elements,
$B_\ordthree$, \emph{blocks}.
%% top, bottom.
It is easily seen that $\bot = \setof{\{\ordfour\}}{\ordfour \in
  \cardone}$ and $\top = \{\cardone\}$; that is, the set of all
singleton subsets of $\cardone$, respectively the singleton set
containing all elements of $\cardone$.

%% <, \prec
As is usual, if $\parti{P}, \parti{Q} \in \thelattice{\cardone}$, we
write $\parti{P} \prec \parti{Q}$ if $\parti{P} < \parti{Q}$ and no
$\parti{R} \in \thelattice{\cardone}$ exists such that
$\parti{P} < \parti{R} < \parti{Q}$.  Furthermore,
$\parti{P} \preceq \parti{Q}$ denotes that either
$\parti{P} \prec \parti{Q}$ or $\parti{P} = \parti{Q}$.
It follows that $\parti{P} \prec \parti{Q}$ if and only if $\parti{Q}$
can be obtained by merging exactly two distinct blocks
of $\parti{P}$. If $\poset{X}$ is a subset of $\thelattice{\cardone}$, we
write $\parti{P} \prec_{\poset{X}} \parti{Q}$ if
$\parti{P}, \parti{Q} \in \poset{X}$ with $\parti{P} < \parti{Q}$, and
there exists no $\parti{R} \in \poset{X}$ such that
$\parti{P} < \parti{R} < \parti{Q}$. A subset
$\poset{X} \subseteq \thelattice{\cardone}$ is called \emph{covering}
if $\parti{P} \prec_{\poset{X}} \parti{Q}$ implies
$\parti{P} \prec \parti{Q}$.

A block $B$ induces an equivalence relation on $\cardone$, defined by
$\eqinpart{\ordthree}{B}{\ordfour}$ if and only if both
$\ordthree,\ordfour\in B$, and a partition $\parti{P}$ naturally
induces an equivalence relation defined by
$\eqinpart{\ordthree}{\parti{P}}{\ordfour}$ if and only if there is a
block $B\in\parti{P}$ with $\eqinpart{\ordthree}{B}{\ordfour}$.
Conversely, any equivalence relation corresponds to the partition
whose blocks are the maximal sets of equivalent
elements. 
This one-to-one correspondence allows us to consider a partition as
its corresponding equivalence relation when convenient, and vice
versa.

If $\parti{P} \in \thelattice{\cardone}$ contains exactly one block
$B$ with $\cardinal{B} ≥ 2$ and the remaining blocks are all
singletons, we call $\parti{P}$ a \emph{singular} partition, following
Ore \cite{Ore42}.

\smallskip

% join and meet
If $\chain{C} \subseteq \thelattice{\cardone}$, then its greatest
lower bound $\glb \chain{C}$ is the partition that satisfies
$\eqinpart{x}{\glb \chain{C}}{y}$ if and only if
$\eqinpart{x}{\parti{P}}{y}$ for all $\parti{P} \in \chain{C}$. That
is, the blocks of $\glb \chain{C}$ are all the nonempty intersections
whose terms are exactly one block from every partition
$\parti{P} \in \chain{C}$.  Conversely, its least upper bound
$\lub \chain{C}$ is the partition such that
$\eqinpart{\ordfour}{\lub \chain{C}}{\ordthree}$ if and only if there
exists a finite sequence of partitions
$\parti{P}^1,\ldots,\parti{P}^{k} \in \chain{C}$ and elements
$\ordtwo^0, \ldots, \ordtwo^{k}, \in \cardone$ such that
$\ordfour = \eqinpart{%
  \eqinpart{%
    \eqinpart{%
      \eqinpart{%
        \ordtwo^0}{\parti{P}^1}{\ordtwo^1}}{\parti{P}^2}{\ordtwo^2}}{\parti{P}^3}{\cdots}}{\parti{P}^k}{\ordtwo^k}
= \ordthree$.
A set of partitions is called {\em complete} if it contains both the
least upper bound and greatest lower bound of all its subset, and {\em
  closed} if this is true for every nonempty subset, i.e., a closed
set need not include $\bot$ and $\top$.

\smallskip

Finally, the {\em cofinality} $\cofinality{\cardone}$ of an infinite
cardinal $\cardone$ is the least cardinal $\cardtwo$ such that a set
of cardinality $\cardone$ can be written as a union of $\cardtwo$ sets
of cardinality strictly smaller than $\cardone$:
$\cofinality{\cardone} = \min\setof{\cardinal{I}} {\cardone =
  \cardinal{\bigcup_{i \in I} A_i} \land \forall i \in I,
  \cardinal{A_i} < \cardone}$.
Since $\cardone = \bigcup_{i \in \cardone}\{i\}$, we always have
$\cofinality{\cardone} ≤ \cardone$.

If $\cofinality{\cardone} = \cardone$, then the cardinal $\cardone$ is
called \emph{regular}, otherwise it is called \emph{singular}.  Under
AC, which is assumed throughout this paper, every infinite successor
cardinal is regular.  K{\"o}nig's Theorem \cite{Konig:1905} implies
$\cofinality{2^\cardone} > \cardone$, and we additionally have
$2^\cardone ≤ 2^\cardtwo$ whenever $\cardone < \cardtwo$.  By Easton's
theorem \cite{Easton:1970}, these are the only two constraints on
permissible values for $2^\cardone$ when $\cardone$ is regular and
when only ZFC is assumed.  In contrast, when the Generalized Continuum
Hypothesis (GCH) is assumed, cardinal exponentiation is completely
determined.

For infinite $\cardone$, $\thelattice{\cardone}$ has cardinality
$2^{\cardone}$ which provides a weak upper bound on the cardinality of
its subsets, in particular maximal chains, maximal antichains, and sets
of complements.

\section{Results}
We summarize here the main contributions of the paper.

\begin{result}[Well-ordered chains: Theorem \ref{thm:wf-all}]
  Let $\cardone$ be any cardinal. The cardinality of a maximal
  well-ordered chain in $\thelattice{\cardone}$ is always exactly
  $\cardone$.
\end{result}

\begin{result}[Long chains: Theorem~\ref{thm:long-nochainreach}]
  Let $\cardone$ be an infinite cardinal.  There exist chains of
  cardinality $> \cardone$ in $\thelattice{\cardone}$.
\end{result}

\begin{result}[Short chains: Theorem~\ref{thm:small_chain}]
  Let $\cardone$ be an infinite cardinal such that for every cardinal
  $\cardtwo < \cardone$ we have $2^\cardtwo < 2^\cardone$. Then there
  exists a maximal chain of cardinality $< 2^{\cardone}$ (but $≥
  \cardone$) in $\thelattice{2^\cardone}$.
\end{result}

\begin{result}[Antichains: Theorems~\ref{thm:antichain_all_card}
  and~\ref{rem:antichain_between_card}]
  Let $\cardone$ be an infinite cardinal. Each non-trivial maximal
  antichain in $\thelattice{\cardone}$ has cardinality between
  $\cardone$ and $2^{\cardone}$, and these bounds are tight (there
  exists maximal antichains with each of these two cardinalities).
\end{result}

\begin{result}[Complements: Theorems~\ref{thm:number_of_complements},
  \ref{cor:compl-existence}, and~\ref{thm:GCHkappa_complements}]
  Let $\cardone$ be an infinite cardinal.
  All cardinals of the form $\cardone^\cardtwo$ with $0 ≤ \cardtwo
  ≤ \cardone$ occur as the number of complements to some partition
  $\parti{P} \in \thelattice{\cardone}$, and these are the only
  cardinalities the set of complements can have.

  For non-trivial partitions $\parti{P} \notin \{\bot,\top\}$, the
  number of complements is between $\cardone$ and $2^\cardone$,
  i.e. $\cardone^{\cardtwo}$ with $1 ≤ \cardtwo ≤ \cardone$.  The
  number of complements to $\parti{P}$ is $2^\cardone$ unless (i)
  $\parti{P}$ contains exactly one block $B$ of cardinality
  $\cardone$, and (ii) $\cardinal{\cardone\setsub B} < \cardone$.  If
  $\parti{P}$ contains one block $B$ of size $\cardone$, then
  $\parti{P}$ has $\cardone^{\cardinal{\cardone \setsub B}}$
  complements.
\end{result}

\begin{result}[Full characterizations under GCH, Theorem \ref{cor:allGCH}]
  Under the Generalized Continuum Hypothesis, when $\cardone$ is an
  infinite cardinal:
  \begin{enumerate}
  \item Any maximal well-ordered chain in $\thelattice{\cardone}$ always
    has cardinality $\cardone$.
  \item Any general maximal chain in $\thelattice{\cardone}$ has cardinality
    \begin{enumerate}
    \item $\cardone^-$, $\cardone$, or $\cardone^+$ (and all three are
      always achieved) if $\cardone$ is a successor cardinal; and
    \item either $\cardone$ or $\cardone^+$ (and both are achieved) if
      $\cardone$ is a limit cardinal.
    \end{enumerate}
  \item Any non-trivial maximal antichain in $\thelattice{\cardone}$ has
    cardinality either $\cardone$ or $\cardone^+$, and both are
    achieved.
  \item Any non-trivial partition has either $\cardone$ or
    $\cardone^+$ complements.  $\parti{P}\notin\{\bot,\top\}$ has
    $\cardone$ complements if and only if (i) $\parti{P}$ contains
    exactly one block, $B$, of cardinality $\cardone$, and (ii)
    $\cardinal{\cardone\setsub B} < \cofinality{\cardone}$; otherwise,
    $\parti{P}$ has $\cardone^+$ complements.
  \end{enumerate}
\end{result}

\section{Some basic properties}
\subsection{Saturated chains in complete lattices}
In this section, we prove a few properties of chains on complete
lattices (not necessarily the partition lattice) that will
be used in the later sections.

\begin{definition}
  \label{def:upper-lower}
  Let $\chain{C}$ be a chain in a complete lattice $\poset{L}$.
  Define the {\em lower}, respectively {\em upper}, subchain relative
  to $x\in\poset{L}$ as
  $\lowerchain{C}{x} = \setof{y\in\chain{C}}{y<x}$ and
  $\upperchain{C}{x} = \setof{y\in\chain{C}}{x<y}$.
\end{definition}

\begin{lemma}
  \label{lem:upperlower}
  Let $\chain{C}$ be a closed chain in a complete lattice $\poset{L}$,
  and $x\in\poset{L}$ such that $\chain{C}\cup \{x\}$ is a chain, and
  the sets $\lowerchain{C}{x}$ and $\upperchain{C}{x}$ are nonempty.
  \begin{itemize}
  \item If $x \notin \chain{C}$, then $\Cminus \prec_{\chain{C}}
    \Cplus$.
  \item If $x \in \chain{C}$, then either (i) $\Cminus = x = \Cplus$,
    (ii) $\Cminus \prec_{\chain{C}} \Cplus$, or (iii)
    $\Cminus \prec_{\chain{C}} x \prec_{\chain{C}} \Cplus$.
  \end{itemize}
  This also implies that when $x \in \chain{C}$,
  $\Cminus \preceq_{\chain{C}} x \preceq_{\chain{C}} \Cplus$.
\end{lemma}

This is easily proved by checking each case. The key points are: i)
Because the chain $\chain{C}$ is closed, there exists no
$y \in \chain{C}$ with $\Cminus < y < \Cplus$; and ii)
$x\in \chain{C} \Rightarrow \Cminus ≤ x ≤ \Cplus$.

\begin{lemma}
  \label{lem:closed-covering-saturated}
  An endpoint-including chain $\chain{C}$ in a complete lattice
  $\poset{L}$ is saturated if and only if it is closed and covering;
  and it is maximal if and only if it is complete and covering.
\end{lemma}

\begin{proof}
  Let in the following $\chain{C}$ be a chain in a complete lattice
  $\poset{L}$, such that $\chain{C}$ has minimal element $c_{\min}$
  and maximal element $c_{\max}$. Since $\chain{C}$ is totally
  ordered, $c_{\min} = \glb\chain{C}$ is its unique least element, and
  $c_{\max} = \lub\chain{C}$ is its unique greatest element.

  \textbf{Saturated implies closed.}  Assume that there exists a
  nonempty subset $\poset{D}\subseteq\chain{C}$ with greatest lower
  bound $\glb\poset{D}\notin\chain{C}$. By construction,
  \begin{equation}\label{Eq:chain}
    c_{\min} = \glb \chain{C} < \glb \poset{D} < \lub \chain{C} =
    c_{\max}
  \end{equation}
  Let $x$ be an arbitrary element of $\poset{C}$. Since $\poset{C}$ is
  a chain, either $d ≤ x$ for some $d \in \poset{D}$, which implies
  $\glb\poset{D} ≤ x$; or $x < d$ for all $d \in \poset{D}$ which
  implies $x ≤ \glb\poset{D}$. It follows that
  $\chain{C} \cup \glb\poset{D}$ is a chain. This together with
  Equation~\eqref{Eq:chain} contradicts that $\chain{C}$ is saturated.
  Dually we prove that $\lub\poset{D} \in \poset{C}$ for every
  nonempty $\poset{D} \subseteq \poset{C}$. Hence $\chain{C}$ is
  saturated only if it is closed.

  \textbf{Saturated implies covering.}  Assume $\chain{C}$ is not
  covering, i.e.~ there exists $x, z \in \chain{C}$ with
  $x \prec_{\chain{C}} z$ but $x \not \prec_{\poset{L}} z$. The second
  relation implies that there exists $y\in\poset{L}\setsub\chain{C}$
  with $x < y < z$.  Because $x \prec_{\chain{C}} z$, every other
  element of $\chain{C}$ is either smaller than $x$ or larger than
  $z$, hence comparable with $y$, whereby $\chain{C}\cup\{y\}$ is a
  chain.  Hence $\chain{C}$ is saturated only if it is covering.

  \textbf{Closed and covering implies saturated.}  Assume that
  $\chain{C}$ is closed and covering, and choose any $x\in\poset{L}$
  with $c_{\min} < x < c_{\max}$ for which $\chain{C}\cup\{x\}$ is
  still a chain. Then $c_{\min}\in\lowerchain C x$ and
  $c_{\max}\in\upperchain C x$, so both sets are nonempty.
  If $x \notin \chain{C}$, Lemma~\ref{lem:upperlower}
  implies $\Cminus \prec_{\chain{C}} \Cplus$, and, because $\chain{C}$
  is covering, $\Cminus \prec \Cplus$.  But $\Cminus ≤ x ≤ \Cplus$, so
  this implies $x = \Cminus$ or $x = \Cplus$, contradicting
  $x \notin \chain{C}$, as both are in
  $\chain{C}$. Hence, $\chain{C}$ is saturated.

  \textbf{Maximal is equivalent to complete and covering.}  A maximal
  chain is a saturated chain that contains $\top$ and $\bot$; and a
  complete sublattice is a closed sublattice that contains $\top$ and
  $\bot$, yielding the lemma's second statement.
\end{proof}

\subsection{Meets and joins of chains in $\thelattice{\cardone}$}
\begin{definition}\label{def:xy-upper-lower}
  Let $\chain{C}$ be a chain in $\thelattice{\cardone}$. We define the
  \emph{lower} and \emph{upper} subchains relative to
  $x, y \in \cardone$ as
  $\lowerchain{C}{x,y} = \setof{\parti{P} \in
    \chain{C}}{\noteqinpart{x}{\parti{P}}{y}}$
  and
  $\upperchain{C}{x,y} = \setof{\parti{P} \in
    \chain{C}}{\eqinpart{x}{\parti{P}}{y}}$;
\end{definition}

\begin{lemma}\label{lem:meet-join-chain}
  Let $\chain{C}$ be a chain in $\thelattice{\cardone}$ and
  $\chain{D}$ be any set of $\cardone$-partitions.
  \begin{enumerate}
  \item $\eqinpart{x}{\glb \chain{D}}{y}$ if and only if
    $\eqinpart{x}{\parti{P}}{y}$ for all $\parti{P} \in \chain{D}$;
  \item $\eqinpart{x}{\lub \chain{C}}{y}$ if and only if
    $\eqinpart{x}{\parti{P}}{y}$ for some $\parti{P} \in \chain{C}$;
  \item[(3)] If $\chain{C}$ is closed, given elements
    $x,y \in \cardone$ for which $\lowerchain{C}{x,y}$ and
    $\upperchain{C}{x,y}$ are both nonempty, we have
    \begin{enumerate}
    \item $\glb \upperchain{C}{x,y} \in\upperchain{C}{x,y}$,
    \item $\lub \lowerchain{C}{x,y} \in\lowerchain{C}{x,y}$,
    \item
      $\lub \lowerchain{C}{x,y} \prec_{\chain{C}}
      \glb\upperchain{C}{x,y}$
      (and, if $\chain{C}$ is covering,
      $\lub\lowerchain{C}{x,y} \prec \glb \upperchain{C}{x,y}$).
    \end{enumerate}
  \end{enumerate}
\end{lemma}

\begin{proof}\
  \begin{enumerate}
  \item This is the definition of $\glb \chain{D}$.
  \item By definition of $\eqinpart{x}{\lub \chain{C}}{y}$, there
    exist {\em finite} sequences $\{x^i\}$ and $\{\parti{P}^i\}$ such
    that $x = \eqinpart{%
      \eqinpart{%
        \eqinpart{%
          \eqinpart{x^0}{\parti{P}^1}{x^1}}%
        {\parti{P}^2}{x^2}}%
      {\parti{P}^3}{\cdots}}%
    {\parti{P}^n}{x^n} = y$.
    But because $\chain{C}$ is totally ordered, the finite set
    $\{\parti{P}^i\} \subset \chain{C}$ has a greatest element
    $\parti{P}^k$, and thus
    $\eqinpart{x_0}{\parti{P}^k}{\eqinpart{x_1}{\parti{P}^k}{\eqinpart{\cdots}{\parti{P}^k}{x_n}}}$,
    whereby $\eqinpart{x}{\parti{P}^k}{y}$.

  \item[(3)] (a) follows immediately from (1), and (b) from the
    negation of (2). For (c), notice that (a) and (b) imply
    $\lub\lowerchain{C}{x,y} \neq \glb\upperchain{C}{x,y}$, whereby
    $\lub\lowerchain{C}{x,y} < \glb\upperchain{C}{x,y}$.  Because
    every $\parti{P}\in\chain{C}$ lies either in $\lowerchain{C}{x,y}$
    or $\upperchain{C}{x,y}$, no $\parti{P}\in\chain{C}$ can have the
    property
    $\lub\lowerchain{C}{x,y} < \parti{P} < \glb\upperchain{C}{x,y}$,
    whereby
    $\lub\lowerchain{C}{x,y} \prec_{\chain{C}}
    \glb\upperchain{C}{x,y}$.
    The final statement is simply the definition of covering.
  \end{enumerate}
\end{proof}

\begin{corollary}\label{cor:lub-blocks-union}
  Given a non-empty chain $\chain{C}$, each block of $\lub \chain{C}$
  is the union, and each block of $\glb\chain{C}$ is the intersection,
  of an increasing sequence of blocks, one from each
  $\parti{P} \in \chain{C}$.
\end{corollary}

\begin{proof}
  Fix $x \in \cardone$, and for each $\parti{P} \in \chain{C}$ let
  $B_{\parti{P}}$ be the block of $\parti{P}$ containing $x$. Let $B$
  (resp. $B'$) be the block of $\glb\chain{C}$ (resp. $\lub\chain{C}$)
  that contain $x$. The second equivalence in each case are
  Lemma~\ref{lem:meet-join-chain}(1-2).
  The statement for the greatest lower bound is derived as
  \[y \in B \Leftrightarrow %
  \eqinpart{x}{\glb\chain{C}}{y} \Leftrightarrow %
  \forall \parti{P} \in \chain{C}, \eqinpart{x}{\parti{P}}{y}
  \Leftrightarrow %
  \forall \parti{P} \in \chain{C}, y \in B_{\parti{P}}
  \Leftrightarrow %
  y \in  \bigcap_{\parti{P}\in\chain{C}} B_{\parti{P}}
  \vspace{-.2cm}
  \]
  and for the least upper bound as
  \[y \in B' \Leftrightarrow %
  \eqinpart{x}{\lub\chain{C}}{y} \Leftrightarrow %
  \exists \parti{P} \in \chain{C}, \eqinpart{x}{\parti{P}}{y}
  \Leftrightarrow %
  \exists \parti{P} \in \chain{C}, y \in B_{\parti{P}}
  \Leftrightarrow %
  y \in  \bigcup_{\parti{P}\in\chain{C}} B_{\parti{P}}
  \vspace{-.2cm}
  \]
\end{proof}

Notice that the choice of $x$, hence also of the $B_{\parti{P}}$ in
the union, is far from unique. These lemmas essentially state that
``nothing happens when going to the limit''. All the equivalences
between elements that are present in the limit (\emph{e.g.}, the join)
were actually already here in some partition of the chain. Notably, if
the chain is well-ordered, the ``merge'' between two elements must
happen between a (partition indexed by an) ordinal and its successor,
and does not suddenly appear at a (partition indexed by a) limit
ordinal.

\subsection{Restriction}
\begin{definition}[Restriction]
  Given a partition $\parti{P}$ of $\cardone$, we define its {\em
    restriction to $\cardtwo\le\cardone$} as
  $\restrict{\parti{P}} =
  \setof{\blockrestrict{B}{\cardtwo}}{B\in\parti{P}} \setsub
  \{\emptyset\}$,
  which is a partition of $\cardtwo$.  Similarly, a set $\poset{D}$ of
  partitions restricts to a subset of $\thelattice{\cardtwo}$ as
  $ \restrict{\poset{D}} =
  \setof{\restrict{\parti{P}}}{\parti{P}\in\poset{D}}$.
\end{definition}

\begin{lemma}\label{lem:restriction} (i)
  $\restrict{\top_{\cardone}} = \top_{\cardtwo}$; %
  (ii) $\restrict{\bot_{\cardone}} = \bot_{\cardtwo}$; %
  (iii) $\parti{P} < \parti{Q}$ implies
  $\restrict{\parti{P}} \le \restrict{\parti{Q}}$; %
  (iv) if $\chain{C}$ is a chain, so is $\restrict{\chain{C}}$; %
  (v) if $\parti{P},\parti{Q}$ are comparable, then
  $\restrict{\parti{P}} < \restrict{\parti{Q}}$ implies
  $\parti{P} < \parti{Q}$; and %
  (vi) $\parti{P} \prec \parti{Q}$ implies
  $\restrict{\parti{P}} \preceq \restrict{\parti{Q}}$.
\end{lemma}

All these Facts can be checked easily by direct applications of the
definitions.

\begin{lemma}%TODO: Something about continuity
  \label{lem:restrict-lub}
  If $\chain{C}$ is a chain in $\thelattice{\cardone}$ and
  $\cardtwo\le \cardone$, then
  $\restrict{(\lub\chain{C})} =
  \lub\left(\restrict{\chain{C}}\right)$,
  and
  $\restrict{(\glb\chain{C})} =
  \glb\left(\restrict{\chain{C}}\right)$.
\end{lemma}

\begin{proof}
  Let $\cardtwo \le \cardone$, and $\chain{C}$ be a chain in
  $\thelattice{\cardone}$.
  If $\chain{C}$ is empty, then the result follows directly
  from Lemma \ref{lem:restriction}(i-ii) and the definition
  of meet and join on the empty set.

  Assume now that $\chain{C}$ is nonempty, and fix $x \in \cardtwo$. For each
  $\parti{P} \in \chain{C}$, let $B_{\parti{P}}$ be the block of
  $\parti{P}$ containing $x$, and $B_{\tilde{\parti{P}}}$ be the block
  of $\tilde{\parti{P}} = \restrict{\parti{P}}$ containing
  $x$.

  \noindent \textbf{Greatest lower bound.} Let $B$ be the block in
  $\glb \chain{C}$, and $B_\cardtwo$ be the block in
  $\glb \left(\restrict{\chain{C}}\right)$, that contains
  $x$. Applying Corollary~\ref{cor:lub-blocks-union} to both
  $B_\cardtwo$ and to $B$ yields
  \[B_\cardtwo = %
  \bigcap_{\tilde{\parti{P}} \in \restrict{\chain{C}}}
  B_{\tilde{\parti{P}}} = %
  \bigcap_{\parti{P} \in \chain{C}} B_{\tilde{\parti{P}}} = %
  \bigcap_{\parti{P} \in \chain{C}} (B_{\parti{P}} \cap \cardtwo) = %
  (\bigcap_{\parti{P} \in \chain{C}} B_{\parti{P}}) \cap \cardtwo = %
  B \cap \cardtwo
  \]
  Note that restriction is not injective, hence we may have
  $\parti{P} \neq \parti{P'}$ but
  $\restrict{\parti{P}} = \restrict{\parti{P'}}$. In this case,
  $B_{\tilde{\parti{P}}} = B_{\tilde{\parti{P'}}}$, thus the second
  equality above is correct.

  As this holds for every $x \in \cardtwo$, each block of
  $\glb \left(\restrict{\chain{C}}\right)$ is the restriction of a
  block of $\glb \chain{C}$, therefore
  $\glb \left(\restrict{\chain{C}}\right) = \restrict{(\glb
    \chain{C})}$.

  \noindent \textbf{Least upper bound:} The proof is symmetrical with
  key equalities being:
  \[B_\cardtwo = %
  \bigcup_{\tilde{\parti{P}} \in \restrict{\chain{C}}}
  B_{\tilde{\parti{P}}} = %
  \bigcup_{\parti{P} \in \chain{C}} B_{\tilde{\parti{P}}} = %
  \bigcup_{\parti{P} \in \chain{C}} (B_{\parti{P}} \cap \cardtwo) = %
  (\bigcup_{\parti{P} \in \chain{C}} B_{\parti{P}}) \cap \cardtwo = %
  B \cap \cardtwo
  \]
\end{proof}

\begin{lemma}
  \label{lem:chain-restrict-wo-compl}
  Restriction preserves well-order and completeness of chains.
\end{lemma}

\begin{proof}
  Let $\cardtwo ≤ \cardone$ and $\chain{C}$ be a chain in
  $\thelattice{\cardone}$. By Lemma~\ref{lem:restriction}(iv),
  $\restrict{\chain{C}}$ is a chain in $\thelattice{\cardtwo}$. Let
  $\chain{D}_{\cardtwo} \subseteq \restrict{\chain{C}}$ be any,
  possibly empty, subset. We have
  $\chain{D}_{\cardtwo} = \restrict{\chain{D}}$, with
  $\chain{D} = \setof{\parti{P} \in \chain{C} }{ \restrict{\parti{P}}
    \in \chain{D}_{\cardtwo}}$.
  Recall that subsets of chains are chains themselves, hence by
  Lemma~\ref{lem:restrict-lub},
  $\lub \chain{D}_{\cardtwo} = \lub \left(\restrict{\chain{D}}\right)
  = \restrict{(\lub\chain{D})}$,
  and
  $\glb \chain{D}_{\cardtwo} = \glb \left(\restrict{\chain{D}}\right)
  = \restrict{(\glb\chain{D})}$.

  \noindent\textbf{Well-order:}
  If $\chain{C}$ is well-ordered, then $\glb \chain{D} \in \chain{D}$
  is its least element. By construction of $\chain{D}$,
  $\glb \chain{D}_{\cardtwo}  =\restrict{\left(\glb \chain{D}\right)}\in \chain{D}_{\cardtwo}$
  which is thus the least element of $\chain{D}_\cardtwo$. Hence, $\restrict{\chain{C}}$ is
  well-ordered.

  \noindent\textbf{Completeness:}
  If $\chain{C}$ is complete, then $\glb \chain{D} \in \chain{C}$.
  Hence,
  $\glb \chain{D}_{\cardtwo} = \restrict{(\glb\chain{D})} \in
  \restrict{\chain{C}}$.
  Similarly, $\lub \chain{D}_{\cardtwo} \in
  \restrict{\chain{C}}$. Thus, $\restrict{\chain{C}}$ is complete.
\end{proof}

\begin{lemma}\label{lem:chain-restrict-max}
  Restriction preserves maximality of chains.
\end{lemma}

\begin{proof}
  Because of the nature of the proof, we explicitly tell in which set
  of partitions the inequalities hold (even when it is the whole
  lattice). We abusively write $\parti{P} <_{\cardtwo} \parti{Q}$ (and
  so on) instead of $\parti{P} <_{\thelattice{\cardtwo}} \parti{Q}$,
  for the sake of clarity.

  Let $\cardtwo ≤ \cardone$ and $\chain{C}$ be a maximal chain in
  $\thelattice{\cardone}$. It is complete and covering by
  Lemma~\ref{lem:closed-covering-saturated}. Hence,
  $\restrict{\chain{C}}$ is a complete chain due to
  Lemma~\ref{lem:chain-restrict-wo-compl} and we only need to show
  that it is also covering.

  Let $\tilde{\parti{P}}, \tilde{\parti{Q}} \in \restrict{\chain{C}}$
  such that
  $\tilde{\parti{P}} \prec_{\restrict{\chain{C}}} \tilde{\parti{Q}}$.
  By construction, there exist $x,y \in \cardtwo$ such that
  $\noteqinpart{x}{\tilde{\parti{P}}}{y}$ and
  $\eqinpart{x}{\tilde{\parti{Q}}}{y}$, whereby
  $\tilde{\parti{P}} \in \lowerchain{(\restrict{C})}{x,y}$ and
  $\tilde{\parti{Q}} \in \upperchain{(\restrict{C})}{x,y}$. Note that
  because both $x$ and $y$ are in $\cardtwo$, we have
  $\lowchain{(\restrict{\chain{C}})}{x,y} =
  \restrict{\lowerchain{C}{x,y}}$
  and
  $\uppchain{(\restrict{\chain{C}})}{x,y} =
  \restrict{\upperchain{C}{x,y}}$.

  By definition of restrictions, there exist (non-unique)
  $\parti{P}, \parti{Q} \in \chain{C}$ such that
  $\tilde{\parti{P}} = \restrict{P}$ and
  $\tilde{\parti{Q}} = \restrict{\parti{Q}}$, and we must have
  $\noteqinpart{x}{\parti{P}}{y}$ and $\eqinpart{x}{\parti{Q}}{y}$,
  whereby $\parti{P} \in \lowerchain{C}{x,y}$ and
  $\parti{Q} \in \upperchain{C}{x,y}$. Since $\chain{C}$ is maximal,
  it is complete and covering, hence
  Lemma~\ref{lem:meet-join-chain}(3)(c) yields
  $\lub\lowerchain{C}{x,y} \prec_{\cardone} \glb\upperchain{C}{x,y}$.
  Because they differ on $x,y\in\cardtwo$, we have
  $\restrict{(\lub\lowerchain{C}{x,y})} \prec_{\cardtwo}
  \restrict{(\glb\upperchain{C}{x,y})}$
  by Lemma~\ref{lem:restriction}(vi), hence
  Lemma~\ref{lem:restrict-lub} yields
  $\lub \left(\restrict{\lowerchain{C}{x,y}}\right) \prec_{\cardtwo}
  \glb \left(\restrict{\upperchain{C}{x,y}}\right)$,
  i.e.,
  $\lub \lowchain{(\restrict{\chain{C}})}{x,y} \prec_{\cardtwo} \glb
  \uppchain{(\restrict{\chain{C}})}{x,y}$.

  By completeness of $\restrict{\chain{C}}$, it contains both
  $\lub \lowchain{(\restrict{\chain{C}})}{x,y}$ and
  $\glb \uppchain{(\restrict{\chain{C}})}{x,y}$, thus
  $\lub \lowchain{(\restrict{\chain{C}})}{x,y}
  \prec_{\restrict{\chain{C}}} %
  \glb \uppchain{(\restrict{\chain{C}})}{x,y}$.
  Because
  $\tilde{\parti{P}} \in \lowchain{(\restrict{\chain{C}})}{x,y}$, we
  have
  $\tilde{\parti{P}} ≤_{\cardtwo} \lub
  \lowchain{(\restrict{\chain{C}})}{x,y}$
  and similarly,
  $\glb \uppchain{(\restrict{\chain{C}})}{x,y} ≤_{\cardtwo}
  \tilde{\parti{Q}}$,

  Because
  $\tilde{\parti{P}} \prec_{\restrict{\chain{C}}} \tilde{\parti{Q}}$
  by definition, we must have $\tilde{\parti{P}} = %
  \lub \lowchain{(\restrict{\chain{C}})}{x,y}
  \prec_{\restrict{\chain{C}}} %
  \glb \uppchain{(\restrict{\chain{C}})}{x,y} = \tilde{\parti{Q}}$,
  therefore $\tilde{\parti{P}} \prec_{\cardtwo} \tilde{\parti{Q}}$ and
  $\restrict{\chain{C}}$ is covering.
  Being both complete and covering, $\restrict{\chain{C}}$ is
  maximal.
\end{proof}

\section{Well-ordered chains in $\thelattice{\cardone}$}
\label{sec:well-ordered}
For finite $\cardone = n$, it is immediate that any maximal chain in
$\thelattice{n}$ has cardinality $n$: Each step reduces the number of
blocks by $1$, whereby going from $\bot$ with $n$ blocks to $\top$
with one block requires $n-1$ steps, hence $n$ elements in the chain.
For $n ≥ 3$, maximal chains are not unique.  In this section, we show
that maximal well-ordered chains in $\thelattice{\cardone}$ {\em
  always} have cardinality $\cardone$, whether $\cardone$ is finite or
infinite.

If a maximal chain in $\thelattice{\cardone}$ is well-ordered of order
type $\ordone$, then $\ordone$ is a successor ordinal.\footnotemark
For clarity, we will write the order type of a maximal chain as
$\ordone+1$ to emphasize that it is a successor ordinal. Because
$\card{\ordone} = \card{\ordone + 1}$, this has no impact on the
cardinality of the chain. Such chains can be written as
$\chain{C} = \setof{\parti{P}_{\ordtwo}}{\ordtwo ≤ \ordone}$ --- or as
$\chain{C} = \setof{\parti{P}_{\ordtwo}}{\ordtwo < \ordone+1}$ to
emphasize the order type --- with $\parti{P}_{0} = \bot$ and
$\parti{P}_{\ordone} = \top$.%
\footnotetext{Any maximal chain in $\thelattice{\cardone}$ has a
  maximal element, namely $\top$.  But since a well-ordered set of
  limit-ordinal type has no maximal element, this implies that the
  order type of a well-ordered maximal chain must be a
  successor-ordinal.}

\begin{lemma}\label{lem:wf-cardinal}
  Let $\cardone$ be any cardinal. If a chain in
  $\thelattice{\cardone}$ is well-ordered of order type $\ordone$,
  then $\card{\ordone} ≤ \cardone$.
\end{lemma}

Note that we are here speaking of any well-ordered chain, not
necessarily a maximal one, so its order type may be anything.

\begin{proof}
  The lemma holds trivially when $\cardone$ is finite, as detailed above. 
  Assume now that $\cardone$ is an infinite cardinal.
  As $\thelattice{\cardone}$ is isomorphic to $\equ{\cardone}$, the
  set of equivalence relations on $\cardone$ ordered by $\subseteq$,
  there is, for any chain of order type $\ordone$ in
  $\thelattice{\cardone}$, a chain of order type $\ordone$ in the
  poset $(\powerset{\cardone \times \cardone},\subseteq)$. Let
  $\chain{C} = \setof{\parti{E}_{\ordtwo}}{\ordtwo < \ordone}$ be such
  a chain and observe that for every $\ordtwo$ with
  $\ordtwo+1 < \ordone$ there exists at least one
  $\ordfour_\ordtwo \in \parti{E}_{\ordtwo + 1}
  \setsub \parti{E}_{\ordtwo}$.
  As $\chain{C}$ is totally ordered under $\subseteq$, this implies
  $\ordfour_\ordtwo\notin\parti{E}_{\ordtwo'}$ when
  $\ordtwo'<\ordtwo$, i.e. the $\ordfour_\ordtwo$ do not repeat.
  Hence,
  $\setof{\ordfour_\ordtwo}{\ordtwo+1 < \ordone} \subseteq \cardone
  \times \cardone$
  has cardinality $\card{\ordone}$, implying
  $\card{\ordone} ≤ \cardinal{\cardone \times \cardone} =
  \cardone$.
\end{proof}

\begin{lemma}\label{lem:wf-cofinality}
  Every well-ordered maximal chain of order type $\ordone+1$ in
  $\thelattice{\cardone}$ satisfies
  $\cofinality{\cardone} ≤ \cardinal{\ordone}$.
\end{lemma}

\begin{proof}
  Let $\chain{C} = \setof{\parti{P}_{\ordtwo}}{\ordtwo < \ordone + 1}$
  be a maximal well-ordered chain in $\thelattice{\cardone}$ of order
  type $\ordone+1$. Consider the partitions in this chain that have at
  least one block of cardinality $\cardone$. Since
  $\top = \parti{P}_{\ordone}$, there is at least one such
  partition. Let $\ordthree$ be the least ordinal such that
  $\parti{P}_{\ordthree}$ contains a block of cardinality
  $\cardone$. It exists, because every non-empty set of ordinals has a
  least element. Let $B^{\ordthree}$ be a block of cardinality
  $\cardone$ in $\parti{P}_{\ordthree}$.

  By Lemma~\ref{lem:upperlower} and maximality of $\chain{C}$ we have
  $\lub\lowerchain{C}{\parti{P}_{\ordthree}}\preceq \parti{P}_{\ordthree}$,
  and $\lub\lowerchain{C}{\parti{P}_{\ordthree}}\in\chain{C}$. That
  is, we can write
  $\parti{P}_{\ordfour} = \lub\lowerchain{C}{\parti{P}_{\ordthree}}$
  either for $\ordfour=\ordthree$ or $\ordfour+1=\ordthree$.  Consider
  for the sake of contradiction that
  $\parti{P}_{\ordfour} \prec \parti{P}_{\ordthree}$.  Then
  $B^{\ordthree}$ would be either a block of $\parti{P}_{\ordfour}$ or
  the union of two such blocks, at least one of them with cardinality
  $\cardone$.  Both cases contradict the hypothesis of $\ordthree$
  being the smallest ordinal for which $\parti{P}_{\ordthree}$
  contains a block of cardinality $\cardone$.  Hence
  $\ordthree=\ordfour$, and
  $\parti{P}_{\ordthree} = \lub\lowerchain{C}{\parti{P}_{\ordthree}} =
  \lub_{\ordtwo < \ordthree} \parti{P}_{\ordtwo}$.

  Corollary~\ref{cor:lub-blocks-union} implies that a sequence
  $\{B^{\ordtwo}\}_{\ordtwo<\ordthree}$ exists such that  
  $B^{\ordthree} = \bigcup_{\ordtwo<\ordthree} B^{\ordtwo}$.  By
  definition of $\ordthree$, we have $\card{B^{\ordtwo}} < \cardone$
  for every $\ordtwo<\ordthree$, and so by the definition of
  co-finality,
  $\cofinality{\cardone} \le \card{\ordthree} \le \card{\ordone}$
\end{proof}

\begin{corollary}\label{cor:wf-regular}
  If $\cardone$ is a regular cardinal, then every well-ordered maximal
  chain in $\thelattice{\cardone}$ indexed by an ordinal $\ordone$
  satisfies $\cardinal{\ordone} = \cardone$.
\end{corollary}

\begin{theorem}\label{thm:wf-all}
  If $\cardone$ is any infinite cardinal, then every well-ordered maximal
  chain of order type $\ordone+1$ in $\thelattice{\cardone}$ satisfies
  $\card{\ordone} = \cardone$.
\end{theorem}

\begin{proof}
  Because of the previous results, we only need to check that 
  $\card{\ordone} ≥ \cardone$ whenever $\cardone$ is singular. 
  Indeed, the
  case of regular $\cardone$ is handled by
  Corollary~\ref{cor:wf-regular} and the upper bound is due to
  Lemma~\ref{lem:wf-cardinal}.

  \noindent\textbf{Cardinal arithmetic:}  It suffices to show that
  $\cardtwo ≤ \card{\ordone}$ for every regular $\cardtwo < \cardone$:
  By standard cardinal arithmetic, the singular $\cardone$ can be
  expressed as the sum of $\cofinality{\cardone}$ smaller successor
  (hence regular) cardinals $\cardtwo_{\ordthree}$. If the above
  inequality holds, then:
  \[
  \cardone = \sum_{\ordthree ≤ \cofinality{\cardone}}
  \cardtwo_{\ordthree} ≤ \sum_{\ordthree ≤
    \cofinality{\cardone}} \card{\ordone} = \cofinality{\cardone}
  \cdot \card{\ordone} = \max(\cofinality{\cardone}, \card{\ordone})
  \]
  And because $\cardone > \cofinality{\cardone}$ by singularity, we
  can finally conclude that $\cardone ≤ \card{\ordone}$.

  \noindent\textbf{Cardinality:}  Let $\chain{C}$ be a maximal, well-ordered
  chain in $\thelattice{\cardone}$.  By
  Lemmas~\ref{lem:chain-restrict-wo-compl}
  and~\ref{lem:chain-restrict-max}, $\restrict{\chain{C}}$ is a
  maximal well-ordered chain in $\thelattice{\cardtwo}$ for every
  regular $\cardtwo<\cardone$. Hence, by
  Corollary~\ref{cor:wf-regular}, we have
  $\cardinal{\restrict{\chain{C}}} = \cardtwo$. Moreover, by
  construction,
  $\cardinal{\restrict{\chain{C}}} ≤ \cardinal{\chain{C}} =
  \card{\ordone}$.
  Thus, for any regular $\cardtwo < \cardone$ we have
  $\cardtwo ≤ \card{\ordone}$ and we can conclude that
  $\cardone ≤ \card{\ordone}$, whereby $\cardone = \card{\ordone}$.
\end{proof}

\begin{remark}
  Notice that any ordinal $\ordone+1$ with $\card{\ordone} = \cardone$
  appears as the order type of some maximal well-ordered chain in
  $\thelattice{\cardone}$. Indeed, the chain of singular partitions
  in $\thelatticeweak{\ordone+1}$ with non-singleton block 
  $B_{\ordtwo} = \setof{\ordthree}{\ordthree ≤ \ordtwo}$ works. In
  other words, it suffices to find a well-order of order type
  $\ordone+1$ on $\cardone$ (which exists by a cardinality argument)
  and add the elements to a single block in this order. 
\end{remark}

\section{Long chains in $\thelattice{\cardone}$}
\label{sec:long-chains}
% diagonal partition.

For any nonempty $S\subseteq\cardone$, we define the partition
\[\diag{S} = \{S\}\cup \setof{\{\ordfour\}}{\ordfour\in\cardone\setsub
  S}.
\]
When $\cardinal{S} ≥ 2$, $\diag{S}$ is the singular partition with
non-singular block $S$.

\begin{remark}
  \label{rem:diagonal-bijection}
  It is easy to verify that $\diag{-}$ is an order isomorphism
  between subsets of $\cardone$ of size $≥ 2$ and the singular
  partitions.  In particular, chains in
  $(\powerset{\cardone},\subseteq)$ that contain only sets of length
  $≥ 2$ are mapped to chains with the same cardinality in
  $(\thelattice{\cardone},≤)$.
\end{remark}

\begin{theorem}\label{thm:long-nochainreach}
  Let $\cardone$ be an infinite cardinal. There is a chain of
  cardinality $> \cardone$ in $\thelattice{\cardone}$. The chain may
  be chosen to be maximal.
\end{theorem}

\begin{proof}
  By a result of Sierpi{\'n}ski \cite{Sierpinski:subsets} (see also
  \cite[Thm.\ 4.7.35]{Harzheim:orderedsets}), there is a chain
  $\chain{D}'$ in the poset $(\powerset{\cardone},\subseteq)$ of
  cardinality $\cardtwo > \cardone$. There is at most one singleton
  element $\{\ordone\}\in\chain{D}'$, and as $\cardtwo$ is infinite,
  $\chain{D} = \chain{D}'\setsub\{\emptyset,\{\ordone\}\}$ is still a
  chain of cardinality $\cardtwo$.

  Then, by Remark~\ref{rem:diagonal-bijection}, the set $\chain{C} =
  \setof{\diag{S}}{S \in \chain{D}}$ is a chain in
  $\thelattice{\cardone}$ of cardinality $\cardtwo>\cardone$.  By the
  Maximal Chain Theorem, any chain in a poset is contained in a
  maximal chain, and the result follows.
\end{proof}

We note that Theorem \ref{thm:long-nochainreach} shows that
under GCH, there are chains whose cardinalities reach the trivial 
upper bound, $2^\cardone$. 
For the special case of $\cardone = \aleph_0$, we may obtain existence
of a maximal chain in $\thelattice{\aleph_0}$ of cardinality
$2^{\aleph_0}$ without use of (G)CH:

\begin{lemma}
  There is a maximal chain of cardinality $2^{\aleph_0}$ in
  $\thelattice{\aleph_0}$.
\end{lemma}

\begin{proof}
  For each $r \in \mathbb{R}$, define the left Dedekind cut
  $D_r = \setof{q \in \mathbb{Q}}{q < r}$ and note that $r < r'$
  implies $D_r \subsetneq D_{r'}$ by density of $\mathbb{Q}$ in
  $\mathbb{R}$. Hence, the set of left Dedekind cuts is a chain in
  $\powerset{\mathbb{Q}}$ of cardinality
  $\cardinal{\mathbb{R}} = 2^{\aleph_0}$, all elements of which have
  cardinality $≥ 2$, and consequently, by
  Lemma~\ref{rem:diagonal-bijection}, the set of partitions
  $\setof{\diag{D_r}}{r \in \mathbb{R}}$ is a chain in
  $\thelattice{\aleph_0}$ of cardinality $2^{\aleph_0}$. By the
  Maximal Chain Theorem, this chain can be extended to a maximal
  chain; note that $2^{\aleph_0}$ is an upper bound on the cardinality
  of any chain in $\thelattice{\aleph_0}$.
\end{proof}

For arbitrary infinite $\cardone$ we do not know whether existence of
a maximal chain in $\thelattice{\cardone}$ of cardinality $2^\cardone$
can be proved without assuming GCH. 

\section{Short chains}
We now turn to the question of whether there are maximal chains of
cardinality strictly less than $\cardone$ in
$\thelattice{\cardone}$. It is immediate that there are no maximal
chains of finite cardinality in $\thelattice{\aleph_0}$. Indeed, each
step in a maximal chain merges exactly two blocks. Hence, starting
from $\bot$ which has infinitely many blocks, after a finite number of
steps it is impossible to reach $\top$ or any other partition that has
only finitely many blocks. For larger cardinals, there is a general
construction that proves existence of short chains. In the special
case where GCH is assumed, it proves existence of a maximal chain of
cardinality $\cardone^-$ for every successor cardinal
$\cardone$. Because this construction is notationally cumbersome, we
first explain the result for the particular case of
$\cardone = 2^{\aleph_0}$, the cardinality of the continuum, in order
to aid the reader's understanding.

\subsection{Short chains in $\thelattice{2^{\aleph_0}}$}\ \\
We build a maximal chain of size $\aleph_0$ in
$\thelattice{2^{\aleph_0}}$. The proof boils down to the fact that
there are countably many binary strings of finite length but
uncountably many binary strings of countably infinite length,
hence an infinite binary tree of depth $\omega$ has uncountably
many leaves but only countably many inner nodes.

The construction proceeds in two steps:
\begin{enumerate}
\item Starting from $\top$, inductively construct a countable chain of
  ``keyframe'' partitions such that the $n^{\text{th}}$ keyframe has
  $2^n$ blocks. The greatest lower bound of this chain will be $\bot$
  with $2^{\aleph_0}$ singleton blocks.
\item Complete the chain into a maximal chain by adding ``inbetween''
  partitions between the keyframes. We will need $2^n-1$ extra
  partitions between the $(n-1)^{\text{st}}$ and the $n^{\text{th}}$
  keyframe.\footnotemark
\end{enumerate}
\footnotetext{``Keyframe'' and ``Inbetween'' are terms from animation:
  A ``keyframe'' is a drawing that defines the start or end of a
  movement. The animation frames between keyframes--the
  ``inbetweens''--are drawn to make the transitions between keyframes
  smooth.}

As a finite number of inbetween partitions are added for each $n$,
only a countable number of partitions are added in total, and so the
full chain has countable length.

\textbf{Keyframes.} Consider partitions of the set $\{0, 1\}^{\omega}$
of countably infinite sequences of bits; note that the cardinality of
$\{0, 1\}^{\omega}$ is $2^{\aleph_0}$. Let $u \in \{0, 1\}^{\omega}$
and let $[u]_k$ be the set of elements of $\{0, 1\}^{\omega}$ with the
same first $k$ bits:
$[u]_k = \setof{u_0u_1\ldots u_{k-1}v}{v \in \{0, 1\}^{\omega}}$.
Define the $k^{\text{th}}$ keyframe to be
$\keyframe{k} = \setof{[u]_k}{u \in \{0, 1\}^{\omega}}$. Similarly,
define $\keyframe{\omega} = \bot$.

Thus, the keyframes are as shown in Figure~\ref{fig:keyframes}. Each level
of the picture corresponds to one partition, with its blocks
depicted. It is instructive to note that the set of blocks of all
partitions form an infinite binary tree and that the number of nodes
in each keyframe is finite (whence the corresponding partition has
only finitely many blocks).  The blocks of $\keyframe{\omega}$
constitute the bottom most level of the tree, and hence are the leaves
of the infinite binary tree, of which there are $2^{\card{\omega}}$.
The total number of \emph{keyframe partitions} corresponds to the 
number of levels in the infinite tree, and is thus clearly
countable. The total number of internal nodes in the tree is
$\sum_{\ordone < \omega} 2^{\card{\ordone}} = \aleph_0$.

\begin{figure}
  \includegraphics[width=\linewidth]{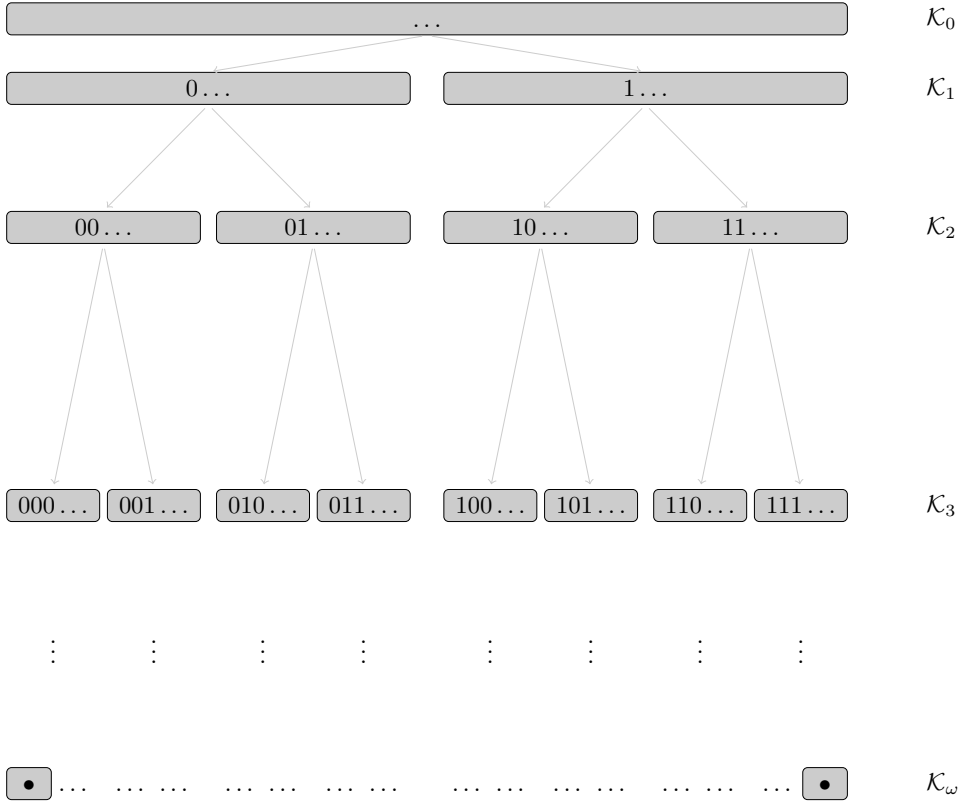}
  \caption{Keyframe partitions}
  \label{fig:keyframes}
\end{figure}

\textbf{Completing the chain.} The set of keyframe partitions clearly
form a chain, but just as clearly this chain is not maximal. Inbetween
partitions must be added between the keyframes to obtain a maximal
chain.

As seen in Figure~\ref{fig:keyframes}, between the $k^{\text{th}}$ and
the $(k+1)^{\text{st}}$ keyframe, $2^k$ blocks have been split. In
order to locally saturate the chain, it suffices to add $2^k-1$ new
partitions, each with one more of the $2^k$ blocks split in 
two. This will ensure that each partition is a successor to the
previous one.

After adding these new partitions, the picture is now as shown in
Figure~\ref{fig:fulltree}. This chain is easily seen to be maximal. At
each (non-keyframe) partition, the new blocks, resulting from
splitting one of the blocks of the parent partition, are shown with
hatching.

\begin{figure}
  \includegraphics[width=\linewidth]{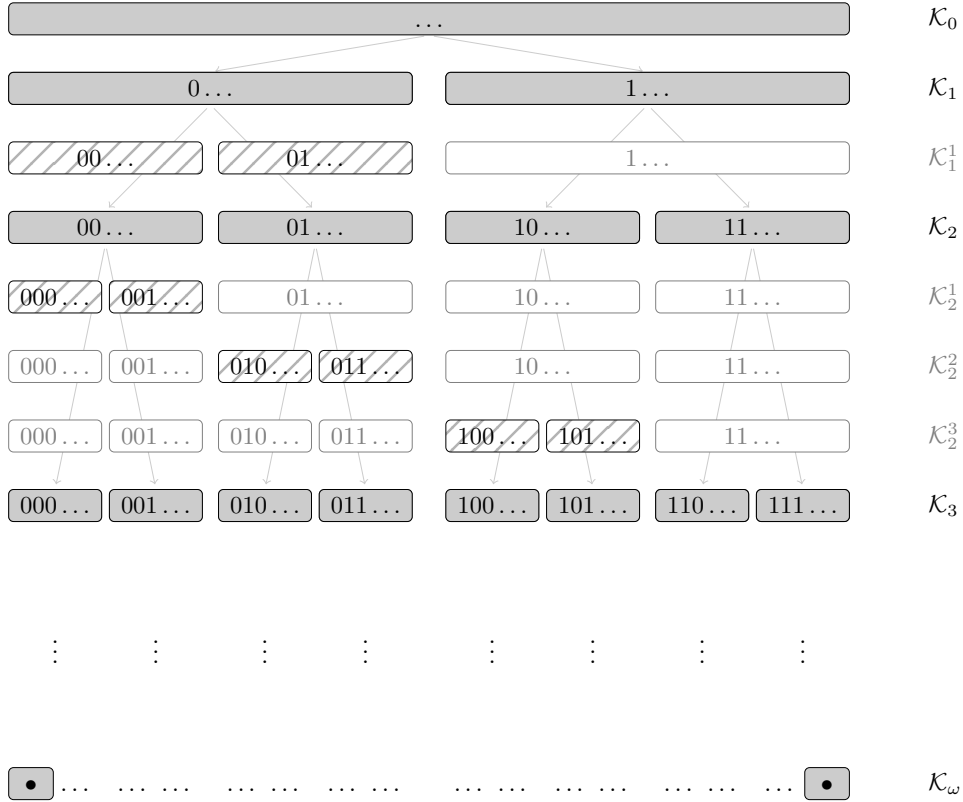}
  \caption{A short maximal chain}
  \label{fig:fulltree}
\end{figure}

Since there are $2^k$ inbetween partitions between the $k^{\text{th}}$
and $(k+1)^{\text{st}}$ keyframes, the total number of inbetween
partitions is $\sum_{k<\omega} 2^k$, which is countable. Since there
are also only countably many keyframes, the total number of partitions
in the chain is countable.

\subsection{Short chains: general construction for large cardinals}
\newcommand{\logcard}{\cardtwo}
\newcommand{\powcard}{\cardone}
\newcommand{\examplecard}{\cardthree}
\def\theord{\order{\logcard}}

Let $\logcard$ be any infinite cardinal and $\powcard =
2^{\logcard}$. 
We will build a maximal chain of cardinality strictly less than $\powcard$ 
in $\thelattice{\powcard} = \thelattice{2^{\logcard}}$.

\textbf{Underlying set.} Let $S$ be the set of binary sequences
indexed by the initial ordinal of $\logcard$. That is $S =
\{0,1\}^{\order{\logcard}}$.  $S$ contains $\powcard$ elements; we shall
use $\thelatticeweak{S}$ as a concrete instance of $\thelattice{\powcard}$.

\textbf{Keyframes.} Let $f \in S$ and $\ordthree ≤ \theord$. We denote
by $[f]_{\ordthree}$ the set of sequences that agree with $f$ on the
initial $\ordthree$ elements:
$$[f]_{\ordthree} = \setof{ g \in S}{\forall \ordtwo < \ordthree,
  f(\ordtwo) = g(\ordtwo) }
$$
Observe that $[f]_{\ordthree}$ is uniquely defined by a binary
sequence indexed by $\ordthree$.  The bracketed notation $[f]$ is used
to emphasize that $[f]$ is truly the equivalence class of $f$ in the
corresponding equivalence, even though everything is written in terms
of partitions.

Let $\keyframe{\ordthree} = \setof{[f]_{\ordthree}}{f \in S}$. Then,
$\keyframe{\ordthree}$ is a partition of $S$ which we call a
\emph{keyframe partition}. From the definition, it follows that
$\keyframe{0} = \top$ and that $\keyframe{\theord} = \bot$.  Note that
$\keyframe{\ordthree}$ consists of $2^{\card{\ordthree}}$ blocks and
is in bijective correspondence with $\{0,
1\}^{\ordthree}$. Let $\chain{K} = \setof{ \keyframe{\ordthree}}{\ordthree ≤ \theord }$.

\begin{lemma}\label{lem:ordo_huge}
  If $\ordone < \ordtwo$ then $[f]_{\ordtwo} \subset [f]_{\ordone}$ 
  and $\keyframe{\ordtwo} < \keyframe{\ordone}$.  Furthermore,
  $\chain{K}$ is a chain in $\thelattice{\powcard}$.
\end{lemma}

\begin{proof}
  The first point is immediate by definition of the refinement
  ordering: if $f$ and $g$ agree on their initial $\ordtwo$ elements,
  then they trivially agree on their initial $\ordone$ elements, so
  $[f]_{\ordtwo} \subsetneq[f]_{\ordone}$.  Since this holds for every
  $f\in S$, we further get $\keyframe{\ordtwo} < \keyframe{\ordone}$.
  The second point is an immediate consequence of the first.
\end{proof}

\begin{lemma}\label{lem:keyframes-closed}
  $\chain{K}$ is closed.
\end{lemma}

\begin{proof}
  Let $D$ be a set of ordinals and
  $\chain{D} = \setof{\keyframe{\ordthree}}{\ordthree \in D}$.  Being
  a set of ordinals, $D$ admits a minimal element $\ordone$. Then
  $\keyframe{\ordone}$ is an upper bound for $\chain{D}$ and being
  part of it, must be the least upper bound.

  Let $\ordtwo$ be the smallest ordinal not in $D$. If
  $\ordtwo = \ordtwo'+1$ is a successor ordinal, then
  $\keyframe{\ordtwo'}$ is both part of $\chain{D}$ and a lower bound,
  hence it's the greatest lower bound.

  Suppose that $\ordtwo$ is a limit ordinal and that
  $\glb \chain{D} > \keyframe{\ordtwo}$. Then, there exists $f, g$
  such that $\eqinpart{f}{\glb \chain{D}}{g}$ but
  $\noteqinpart{f}{\keyframe{\ordtwo}}{g}$. By definition of
  $\keyframe{\ordtwo}$, there exists $\ordfour < \ordtwo$ with
  $f(\ordfour) \neq g(\ordfour)$. However, by definition of $\ordtwo$,
  there must exists $\ordfour' \in D$ with
  $\ordfour ≤ \ordfour' < \ordtwo$. $f(\ordfour) \neq g(\ordfour)$
  thus implies $\noteqinpart{f}{\keyframe{\ordfour'}}{g}$ which
  contradicts $\eqinpart{f}{\glb \chain{D}}{g}$. Therefore,
  $\chain{K}$ is closed.
\end{proof}

\newcommand{\ordtwotod}{\order{2^{\card{\ordthree}}}}

\textbf{Locally saturating the chain.} We fix $\ordthree < \theord$
and construct a saturated chain bounded below by
$\keyframe{\ordthree+1}$ and above by $\keyframe{\ordthree}$.

$\keyframe{\ordthree}$ has cardinality $2^{\card{\ordthree}}$.  Fix,
by the Axiom of Choice, a well-ordering $<_{\ordthree}$ of
$\keyframe{\ordthree}$ with order type $\ordtwotod$. For
$0 ≤ \ordone < \ordtwotod$, let the block $S_{\ordthree, \ordone}$ be
the $\ordone\supth$ element of $\keyframe{\ordthree}$ according to
this ordering.
% JYM:
% "of order type bla" means everything strictly smaller than bla. eg,
% of order type omega means the natural numbers. Thus we need to add
% the formal element with the limit ordinal to make everything work
% smoothly.
%
% with finite ordinals, we have, say S_2 with 4 blocks (check the
% figure of the tree) numbered 0, 1, 2, 3 which is of order type
% 4. But to go from S_2 to S_3, we need S_2^0 = S_2, S_2^1 (one block
% split), S_2^2, S_2^3 and S_2^4 = S_3 (four blocks splits).

For each $0 ≤ \ordone < \ordtwotod$, define the set
$[f]_{\ordthree}^{\ordone}$ as
$$[f]_{\ordthree}^{\ordone} = \left\{
  \begin{array}{ll}
    [f]_{\ordthree + 1} & \text{if } [f]_{\ordthree}
    <_{\ordthree} S_{\ordthree, \ordone}\\
% Next line voluntarily left empty to avoid "\\[f]" interpertation.

    [f]_{\ordthree} & \text{if } S_{\ordthree, \ordone}
    ≤_{\ordthree} [f]_{\ordthree}
  \end{array}
\right.
$$
and let $\intermediate{\ordone} = \setof{ [f]_{\ordthree}^{\ordone}}{f
  \in S}$.  Observe that the sets $[f]_{\ordthree}^{\ordone}$ are
pairwise disjoint and that their union is $S$. Thus,
$\intermediate{\ordone}$ is a partition of $S$ (an ``inbetween''
partition), obtained by splitting the initial $\ordone$ blocks of
$\keyframe{\ordthree}$.
By construction there are $2^{\card{\ordthree}}$ inbetween partitions
between $\keyframe{\ordthree}$ and $\keyframe{\ordthree+1}$.

\begin{lemma}\label{lem:initial_intermediate}
  For all $f \in S$, $[f]_{\ordthree}^0 = [f]_{\ordthree}$. Thus
  $\intermediate{0} = \keyframe{\ordthree}$.
\end{lemma}

\begin{proof}
  By construction, $S_{\ordthree, 0}$ is minimal for $<_{\ordthree}$,
  thus it is impossible for $[f]_{\ordthree}$ to be strictly smaller.
\end{proof}

\begin{lemma}\label{lem:intermediate_split}
  Let $\ordone < \ordtwo$. Then for each $f$,
  $[f]_{\ordthree}^{\ordtwo} \subseteq [f]_{\ordthree}^{\ordone}$.
\end{lemma}

\begin{proof} By cases depending on $<_{\ordthree}$. First note that
  $\ordone < \ordtwo$ implies $S_{\ordthree, \ordone} <_{\ordthree}
  S_{\ordthree, \ordtwo}$ by construction.
  \begin{enumerate}
  \item If $[f]_{\ordthree} <_{\ordthree} S_{\ordthree, \ordone}
    <_{\ordthree} S_{\ordthree, \ordtwo}$ then
    $[f]_{\ordthree}^{\ordtwo} = [f]_{\ordthree+1} =
    [f]_{\ordthree}^{\ordone}$.
  \item If $S_{\ordthree, \ordone} ≤_{\ordthree} [f]_{\ordthree}
    <_{\ordthree} S_{\ordthree, \ordtwo}$ then
    $[f]_{\ordthree}^{\ordtwo} = [f]_{\ordthree+1} \subsetneq
    [f]_{\ordthree} = [f]_{\ordthree}^{\ordone}$.
  \item If $ S_{\ordthree, \ordone} <_{\ordthree} S_{\ordthree, \ordtwo}
    ≤_{\ordthree} [f]_{\ordthree}$ then $[f]_{\ordthree}^{\ordtwo} =
    [f]_{\ordthree} = [f]_{\ordthree}^{\ordone}$.
  \end{enumerate}\ \\[-2em]
\end{proof}

\begin{lemma}\label{lem:ordo_large}
  Let $\ordone < \ordtwo$. Then $\intermediate{\ordtwo} <
  \intermediate{\ordone}$.
\end{lemma}

\begin{proof}
  By Lemma \ref{lem:intermediate_split}, each block of
  $\intermediate{\ordtwo}$ is contained in a block of
  $\intermediate{\ordone}$, thus $\intermediate{\ordtwo} ≤
  \intermediate{\ordone}$.

  Since $\ordone < \ordtwo$, there exists $f$ such that $S_{\ordthree,
    \ordone} ≤_{\ordthree} [f]_{\ordthree} <_{\ordthree}
  S_{\ordthree, \ordtwo}$. Hence there are blocks of
  $\intermediate{\ordtwo}$ that are strictly contained in blocks of
  $\intermediate{\ordone}$, and thus the inequality is indeed strict.
\end{proof}

Let $\chain{C}_{\ordthree}' = \setof{\intermediate{\ordone}}{0 ≤ \ordone < \ordtwotod}$
  and
  $\chain{C}_{\ordthree} = \chain{C}_{\ordthree}' \cup
  \{\keyframe{\ordthree+1}\}$.

\begin{corollary}
  \label{cor:chain-chunk}
  $\chain{C}_{\ordthree}$ is a chain with minimum
  $\keyframe{\ordthree+1}$ and maximum
  $\keyframe{\ordthree} = \intermediate{0}$.
\end{corollary}

Note that including $\keyframe{\ordthree+1}$ in
$\chain{C}_{\ordthree}'$ (that is, defining only
$\chain{C}_{\ordthree}$) appears to be more natural, but because it is
both an element of $\chain{C}_{\ordthree}$ and of
$\chain{C}_{\ordthree+1}$, it becomes notationally awkward when taking
unions of chains and computing the cardinality of such unions.

\begin{lemma}\label{lem:inbetweens-closed}
  $\chain{C}_{\ordthree}$ is closed.
\end{lemma}

\begin{proof}
  Let $D$ be a set of ordinals and
  $\chain{D} = \setof{\intermediate{\ordone}}{\ordone \in D}$.  Being
  a set of ordinals $D$ admits a minimal element $\ordone'$.  Then
  $\intermediate{\ordone'}$ is an upper bound for $\chain{D}$ and
  being part of it, must be the least upper bound.

  Let $\ordtwo$ be the smallest ordinal not in $D$. If
  $\ordtwo = \ordtwo'+1$ is a successor ordinal, then
  $\intermediate{\ordtwo'}$ is both part of $\chain{D}$ and a lower
  bound, hence it's the greatest lower bound.

  Suppose that $\ordtwo$ is a limit ordinal and that
  $\keyframe{\ordthree+1} ≤ \intermediate{\ordtwo} < \glb \chain{D} ≤
  \keyframe{\ordthree}$.
  Then, there exists $f, g$ such that
  $\noteqinpart{f}{\intermediate{\ordtwo}}{g}$ but
  $\eqinpart{f}{\glb \chain{D}}{g}$ hence
  $\eqinpart{f}{\keyframe{\ordthree}}{g}$. The later implies
  $[f]_{\ordthree} = [g]_{\ordthree}$, therefore we must have
  $[f]_{\ordthree} <_{\ordthree} S_{\ordthree, \ordtwo}$ in order to
  have
  $[f]_{\ordthree}^{\ordtwo} = [f]_{\ordthree+1} \neq
  [g]_{\ordthree+1} = [g]_{\ordthree}^{\ordtwo}$.

  Now, because $\ordtwo$ is limit, there exists $\ordfour$ with
  $[f]_{\ordthree} <_{\ordthree} S_{\ordthree, \ordfour} <
  S_{\ordthree, \ordtwo}$
  and by definition of $\ordtwo$, we can choose $\ordfour \in
  D$.
  However, this implies
  $[f]_{\ordthree}^{\ordfour} = [f]_{\ordthree+1} \neq
  [g]_{\ordthree+1} = [g]_{\ordthree}^{\ordfour}$
  and thus $\noteqinpart{f}{\intermediate{\ordfour}}{g}$,
  contradicting $\eqinpart{f}{\glb \chain{D}}{g}$. Therefore,
  $\chain{C}_{\ordthree}$ is closed.
\end{proof}

For any $f \in S$ let $f_0 \in [f]_{\ordthree}$ (resp. $f_1\in
[f]_{\ordthree}$) be the binary sequence such that $f_0(\ordthree') =
0$ (resp. $f_1(\ordthree') = 1$) for all $\ordthree' ≥
\ordthree$. Notice that $[f_0]_{\ordthree} = [f_1]_{\ordthree}$ but
$[f_0]_{\ordthree+1} \neq [f_1]_{\ordthree+1}$.

For any $g \in [f]_{\ordthree}$, either $g(\ordthree) = 0$ and
$[g]_{\ordthree+1} = [f_0]_{\ordthree+1}$, or $g(\ordthree) = 1$ and
$[g]_{\ordthree+1} = [f_1]_{\ordthree+1}$. Thus, $[f]_{\ordthree} =
[f_0]_{\ordthree+1} \cup [f_1]_{\ordthree+1}$.

\begin{lemma}
  \label{lem:keyframe-order}
  For every $\ordone < \ordtwotod$, $\intermediate{\ordone+1} \prec
  \intermediate{\ordone}$.
\end{lemma}

\begin{proof}
  Fix $\ordthree$ and $\ordone$. Let $f \in S$ and consider the block
  $[f]_{\ordthree}^{\ordone}$ of $\intermediate{\ordone}$. There are
  three cases:
  \begin{enumerate}
  \item If $[f]_{\ordthree}<_{\ordthree} S_{\ordthree, \ordone}
    <_{\ordthree} S_{\ordthree, \ordone+1}$ then
    $[f]_{\ordthree}^{\ordone} = [f]_{\ordthree}^{\ordone+1} =
    [f]_{\ordthree+1}$

  \item If $S_{\ordthree, \ordone} <_{\ordthree} S_{\ordthree,
        \ordone+1} ≤_{\ordthree} [f]_{\ordthree}$ then
    $[f]_{\ordthree}^{\ordone} = [f]_{\ordthree} =
    [f]_{\ordthree}^{\ordone+1}$.

  \item If
    $ [f]_{\ordthree} = S_{\ordthree, \ordone}<_{\ordthree}
    S_{\ordthree, \ordone+1}$
    then either $[f]^{\ordone+1}_{\ordthree} = [f_0]_{\ordthree+1}$ or
    $[f]^{\ordone+1}_{\ordthree} = [f_1]_{\ordthree+1}$.
  \end{enumerate}

  (1) and (2) follow immediately from the definition. For case (3), we
  first obtain $[f]^{\ordone}_{\ordthree} = [f]_\ordthree$ and
  $[f]^{\ordone+1}_\ordthree = [f]_{\ordthree+1}$ by definition. But
  since
  $[f]_\ordthree = [f_0]_{\ordthree+1} \cup [f_1]_{\ordthree+1}$, we
  must have either $f\in [f_0]_{\ordthree+1}$ or
  $f\in [f_1]_{\ordthree+1}$.  Hence
  $[f]^{\ordone+1}_\ordthree = [f]_{\ordthree+1}$ is either
  $[f_0]_{\ordthree+1}$ or $[f_1]_{\ordthree+1}$, as desired. By the
  above, every block of $\intermediate{\ordone}$ is a block of
  $\intermediate{\ordone+1}$ except for a single block (the unique
  block for which $[f]_{\ordthree} = S_{\ordthree, \ordone}$), which
  is obtained by taking the union of exactly two blocks of
  $\intermediate{\ordone+1}$.  Hence,
  $\intermediate{\ordone+1} \prec \intermediate{\ordone}$.
\end{proof}

\begin{lemma}\label{lem:inbetween-covering}
  $\chain{C}_{\ordthree}$ is covering.
\end{lemma}

\begin{proof}
  Let $\parti{P} \prec_{\chain{C}_{\ordthree}} \parti{Q}$ be two
  partitions of $\chain{C}_{\ordthree}$. By definition, we must have
  $\parti{Q} = \intermediate{\ordone}$ for some $\ordone$, therefore
  $\parti{P} = \intermediate{\ordone+1}$. By the previous lemma, we
  know that $\intermediate{\ordone+1} \prec \intermediate{\ordone}$,
  hence $\chain{C}_{\ordthree}$ is covering.
\end{proof}

\begin{proposition}\label{prop:maximal_intermediate}
  $\chain{C}_{\ordthree}$ is a saturated chain in
  $\thelattice{\powcard}$, with minimum $\keyframe{\ordthree+1}$ and maximum
  $\keyframe{\ordthree}$.
\end{proposition}

\begin{proof}
  Applying previous results, we find that $\chain{C}_{\ordthree}$ is
  an endpoint-including (Corollary~\ref{cor:chain-chunk}), closed
  (Lemma~\ref{lem:inbetweens-closed}) and covering
  (Lemma~\ref{lem:inbetween-covering}) chain. Therefore, by
  Lemma~\ref{lem:closed-covering-saturated}, it is saturated.
\end{proof}

\textbf{Completing the chain.} In the following, let $\chain{C} =
\{\bot\} \cup \left( \bigcup_{0 ≤ \ordthree < \theord}
  \chain{C}_{\ordthree}' \right)$.

\begin{theorem}\label{the:maximal_chain_C}
  $\chain{C}$ is a maximal chain in $\thelattice{\powcard}$.
\end{theorem}

\begin{proof}
  By Proposition \ref{prop:maximal_intermediate},
  $\chain{C}_{\ordthree} = \chain{C}_{\ordthree}' \cup
  \{\keyframe{\ordthree+1}\}$ is a saturated chain with minimum
  $\keyframe{\ordthree+1}$ and maximum $\keyframe{\ordthree}$, whereby
  each $\chain{C}_{\ordthree}'$ is a chain. Noting that $\ordthree <
  \ordfour$ implies $\keyframe{\ordfour} < \keyframe{\ordthree}$ by
  Lemma \ref{lem:ordo_huge}, every element of $\chain{C}_{\ordthree}'$
  is comparable to every element of $\chain{C}_{\ordfour}'$. Hence,
  $\chain{C}$ is a chain.

    Suppose that there exists a partition
    $\parti{P} \notin \chain{C}$ such that
    $\chain{C} \cup \{\parti{P}\}$ is still a chain. We obviously have
    $\parti{P} \notin \chain{K} \subset \chain{C}$ and because
    $\chain{K}$ is closed (Lemma~\ref{lem:keyframes-closed}) we can
    apply Lemma~\ref{lem:upperlower} and conclude that
    $\lub \lowerchain{K}{\parti{P}} \prec_{\chain{K}} \glb
    \upperchain{K}{\parti{P}}$.
    Because they are both keyframes, we must have
    $\glb \upperchain{K}{\parti{P}} = \keyframe{\ordtwo}$ for some
    ordinal, hence
    $\lub \lowerchain{K}{\parti{P}} = \keyframe{\ordtwo+1}$. This
    entails $\keyframe{\ordtwo+1} < \parti{P} < \keyframe{\ordtwo}$,
    but we now by Proposition~\ref{prop:maximal_intermediate} that the
    chain is saturated between these two. Hence, $\chain{C}$ is
    maximal.
\end{proof}

\textbf{Length of the chain.} Recall that $\chain{C}_{\ordthree}'$ has
cardinality $2^{\card{\ordthree}}$.

\begin{lemma}\label{lem:chain-length}
  $\chain{C}$ has cardinality $\sum_{\ordthree < \theord}
  2^{\card{\ordthree}}$.
\end{lemma}

\begin{proof}
  $\cardinal{\chain{C}} = \cardinal{\bigcup_{\ordthree < \theord}
    \chain{C}_{\ordthree}' \cup \{\bot\}} = \cardinal{\bigcup_{
      \ordthree < \theord} \chain{C}_{\ordthree}'} = \sum_{\ordthree <
    \theord} \cardinal{\chain{C}_{\ordthree}} = \sum_{\ordthree <
    \theord} 2^{\card{\ordthree}}$
\end{proof}

\begin{lemma}\label{lem:use-Konig}
  Let $\logcard$ be a cardinal such that for any
  $\examplecard < \logcard$, we have $2^\examplecard <
  2^\logcard$.
  Then
  $\sum_{\ordthree < \order{\logcard}} 2^{\card{\ordthree}} <
  2^\logcard$.
\end{lemma}

\begin{proof}
If $\ordthree < \order{\logcard}$, then
$\cardinal{\ordthree} < \logcard$, thus by hypothesis,
$2^{\cardinal{\ordthree}} < 2^{\logcard}$.

Assume now, for contradiction, that
$\sum_{\ordthree < \order{\logcard}} 2^{\card{\ordthree}} ≥
2^{\logcard}$. 
Recall that
\[
\cofinality{2^\logcard} = \inf\setof{\cardinal{I}}{2^\logcard =
  \cardinal{\bigcup_{i \in I} A_i} \glb \forall i \in I,
  \cardinal{A_i} < 2^\logcard}
\]
which is the same as
$\cofinality{2^\logcard} = \inf\setof{\cardinal{I}}{2^\logcard =
  \sum_{i \in I} \examplecard_i \glb \forall i \in I, \examplecard_i <
  2^\logcard}$.

  It then follows that $\cofinality{2^{\logcard}} ≤
  \cardinal{\order{\logcard}} = \logcard$. But by a standard
  consequence of K{\"o}nig's Theorem \cite{Konig:1905} we always have
  $\cofinality{2^\logcard} > \logcard$ under the Axiom of Choice.
\end{proof}

We can now prove our main result on short chains.

\begin{theorem}\label{thm:small_chain}
  Let $\logcard$ be an infinite cardinal such that for every cardinal
  $\examplecard < \logcard$ we have $2^\examplecard < 2^\logcard$.  Then there
  exists a maximal chain of cardinality $< 2^{\logcard}$ (but $≥
  \logcard$) in $\thelattice{2^\logcard}$.
\end{theorem}

\begin{proof}
  By Theorem \ref{the:maximal_chain_C} and
  Lemma~\ref{lem:chain-length}, there is a maximal chain in
  $\thelattice{2^\logcard}$ of cardinality 
  $\sum_{\ordthree
    < \order{\logcard}}
  2^{\card{\ordthree}}$. By Lemma \ref{lem:use-Konig} we have
  $\sum_{\ordthree
    < \order{\logcard}} 2^{\card{\ordthree}} < 2^\logcard$.

  Consequently, the chain is of cardinality $<
  2^\logcard$ (but $≥ \logcard$).
\end{proof}
For the case of $\logcard = \aleph_0$, we trivially have
$2^n < 2^{\aleph_0}$ for all $n < \aleph_0$, and thus there is a
countable maximal chain in $\thelattice{2^{\aleph_0}}$.
The condition is also satisfied for all strong limit cardinals.

\section{Antichains and maximal antichains}
An \emph{antichain} in $(\poset{P},≤)$ is a subset $\chain{A}
\subseteq \poset{P}$ in which no two distinct elements of $\poset{P}$
are $≤$-comparable. An antichain is \emph{maximal} if adding an
element to it results in a set that is not an antichain.  Observe that
in a poset, the \emph{trivial} antichains $\{\bot\}$ and $\{\top\}$
are always maximal antichains.

There is no known tight bound on the cardinality of maximal antichains
in $\thelattice{n}$ for finite $n$, but some asymptotic results are
known \cite{Canfield1998188,BlinovskyHarper:antichain}.
The cardinality of a maximal antichain in $\thelattice{n}$ is
$\Theta\left(n^{a}(\log n)^{-a-1/4} S(n,K_n)\right)$ where
$a = (2 - e\log2)/4$ and $S(n,K_n) = \max_k\stirling{n}{k}$ is the
largest Stirling number of the second kind for fixed $n$.

\begin{theorem}\label{thm:antichain_all_card}
  Let $\cardone$ be infinite. There is a maximal antichain of
  cardinality $\cardone$ and a maximal antichain of cardinality
  $2^{\cardone}$ in $\thelattice{\cardone}$.
\end{theorem}

\begin{proof}
  Atoms of $\thelattice{\cardone}$ are singular partitions with the
  non-singleton block containing only two elements and are thus in
  bijection with the set of two-element subsets of $\ordone$, hence
  there are $\cardone$ atoms.  By atomicity of
  $\thelattice{\cardone}$, atoms are mutually incomparable (they form
  an antichain) and every other non-$\bot$ element lies over an atom
  (the antichain is maximal).

  Similarly, co-atoms are two blocks partitions of the form
  $\{A, \cardone \setminus A\}$. By picking the subset not containing
  $0$, they are in bijection with non-empty subsets of
  $\cardone \setminus \{0\}$, hence there are $2^{\cardone}$ co-atoms.
  By co-atomicity of $\thelattice{\cardone}$, co-atoms form a maximal
  antichain.
\end{proof}

Since $\thelattice{\cardone}$ itself has cardinality $2^\kappa$, there
is no antichain of a greater size. The proof that there is no
maximal antichain shorter than $\kappa$ is more invoved.

\begin{theorem}\label{thm:antichain_tight_bounds}
  For infinite $\cardone$, no non-trivial maximal antichain in
  $\thelattice{\cardone}$ has cardinality less than $\cardone$.
\end{theorem}

\begin{proof}
  Let $\cardtwo < \cardone$ and
  $\chain{A} = \setof{\parti{A}_{\ordthree}}{\ordthree <
    \order{\cardtwo}}$
  be a non-trivial antichain of cardinality $\cardtwo$ in
  $\thelattice{\cardone}$. We will show that $\chain{A}$ is not
  maximal by building a partition $\parti{P}\notin\chain{A}$ that is not comparable 
  to any of its elements.

  \smallskip

  We first consider the case where $\cardtwo$ is \emph{infinite}.

  \textbf{Small partitions.} For each $\parti{A}_{\ordthree}$, we
  build the set $S_{\ordthree}$ of all the elements that are not in a
  singleton block in $\parti{A}_{\ordthree}$: 
  \[
  S_{\ordthree} = \setof{x \in \cardone}{\text{There exists $y \neq x$
      with $\eqinpart{x}{\parti{A}_{\ordthree}}{y}$}}
  \]
  Now, we call a partition $\parti{A}_{\ordthree}$ \emph{small} if
  $\cardinal{S_{\ordthree}} ≤ \cardtwo$, i.e.~there are less 
  than $\cardtwo$ elements in all the non-singleton blocks of
  $\parti{A}_{\ordthree}$. Given an antichain $\chain{A}$, either it
  contains some small partitions or not.

  \textbf{No small partitions.} Suppose that $\chain{A}$ contains no
  small partitions. Since it is not trivial, we know that 
  $\top \notin \chain{A}$. Hence, for each
  $\parti{A}_{\ordthree} \in \chain{A}$ we can pick an element
  $x_{\ordthree}$ which is not in the same block as $0$, i.e.~
  $\noteqinpart{0}{\parti{A}_{\ordthree}}{x_{\ordthree}}$. Note that
  the $x_{\ordthree}$ are not necessarily distinct. Now, build the
  singular partition $\parti{P}$ whose non-singleton block is
  $P_0 = \{0\} \cup \setof{x_{\ordthree}}{\ordthree <
    \order{\cardtwo}}$.
  We claim that $\parti{P} \notin \chain{A}$ but
  $\chain{A} \cup \{\parti{P}\}$ is still an antichain, hence
  $\chain{A}$ is not maximal.

  Because $P_0$ has cardinality at most $\cardtwo$ and all the remaining blocks are 
  singletons, $\parti{P}$ is small and hence not in $\chain{A}$ by assumption.
  Moreover, it is not possible that
  $\parti{A}_{\ordthree} ≤ \parti{P}$. Indeed, because $\parti{P}$
  contains only one non-singleton block, $P_0$, the comparison can
  hold only if $P_0$ contains $S_{\ordthree}$, the union of all the
  non-singleton blocks of $\parti{A}_{\ordthree}$. But $S_{\ordthree}$
  has cardinality larger than $\cardtwo$ (because
  $\parti{A}_{\ordthree}$ is not small) while $P_0$ has cardinality at
  most $\cardtwo$. Hence, $S_{\ordthree} \not \subset P_0$ and
  $\parti{A}_{\ordthree} \not ≤ \parti{P}$.

  Next, for all $\ordthree < \order{\cardtwo}$ we have 
  $\eqinpart{0}{\parti{P}}{x_{\ordthree}}$ by construction, but
  $\noteqinpart{0}{\parti{A}_{\ordthree}}{x_{\ordthree}}$, so
  $\parti{P} \not ≤ \parti{A}_{\ordthree}$. Thus, $\chain{A}\cup\{\parti{P}\}$
  is an antichain, and hence $\chain{A}$ was not maximal.

  % So, if $\chain{A}$ contains no small partition, we can build a
  % partition $\parti{P}$ that can expand the antichain, hence
  % $\chain{A}$ is not a maximal antichain.

  \textbf{With small partitions.} We now look at the case where
  $\chain{A}$ does contain small partitions. Let $Δ$ be the set of
  indices of small partitions:
  $Δ = \setof{\ordthree}{\cardinal{S_{\ordthree}} ≤ \cardtwo}$.  Note
  that by construction, since $\cardinal{\chain{A}} = \cardtwo$, we
  have $\cardinal{Δ} ≤ \cardtwo$. Finally, let $S$ be the union of all
  non-singleton blocks among all the small partitions in $\chain{A}$:
  $S = \bigcup_{\ordthree \in Δ} S_{\ordthree}$.

  It follows from cardinal arithmetic that
  $\cardinal{S} ≤ \cardtwo$: 
  \[
    \cardinal{S} = \cardinal{\bigcup_{\ordthree \in Δ} S_{\ordthree}}
    ≤ \sum_{\ordthree \in Δ} \cardinal{S_{\ordthree}} ≤
    \sum_{\ordthree \in Δ} \cardtwo = \cardinal{Δ} \times \cardtwo ≤
    \cardtwo \times \cardtwo = \cardtwo
  \]

  The construction proceeds as follows:
  Since
  $\cardinal{S} ≤ \cardtwo < \cardone$, we have
  $\cardinal{\cardone \setminus S} = \cardone$. Thus, in
  $\cardone \setminus S$, we can pick $\cardtwo$ different elements
  $y_{\ordthree}$, one for each $\ordthree < \order{\cardtwo}$. 
  We claim that for each $\ordthree < \order{\cardtwo}$ we can find a
  $x_{\ordthree} \in S$ with
  $\noteqinpart{x_{\ordthree}}{\parti{A}_{\ordthree}}{y_{\ordthree}}$.
  Indeed, if that were not the case then the block of
  $\parti{A}_{\ordthree}$ containing $y_{\ordthree}$ would also
  contain the entirety of $S$ and thus any of the small partitions in 
  $\chain{A}$ would be smaller than $\parti{A}_{\ordthree}$,
  contradicting the fact that $\chain{A}$ is an antichain.

  Note that the $y_{\ordthree}$ are distinct by hypothesis and that
  $x_{\ordthree} \neq y_{\ordthree}$, because the former is in $S$ and
  the latter is not. However, it is possible that the $x_{\ordthree}$ are
  not distinct.

  We now build the equivalence (and partition) $\parti{Q}$ generated
  by the relation
  $\setof{(x_{\ordthree}, y_{\ordthree})}{\ordthree <
    \order{\cardtwo}}$.
  Because the $x_{\ordthree}$ are not necessarily distinct, $\parti{Q}$ may 
  contain blocks with more than two elements; indeed transitivity may
  link some of the pairs together. However, because  
  the $y_{\ordthree}$ {\em are} distinct, the separation of every $x_{\ordthree}$ and 
  $y_{\ordthree}$ ensures that we know each non-singleton block of
  $\parti{Q}$ to contain exactly one $x_{\ordthree}$ and one or more
  $y_{\ordthree}$. 

  For every $\ordthree < \order{\cardtwo}$ we have by construction
  $\eqinpart{x_{\ordthree}}{\parti{Q}}{y_{\ordthree}}$, but
  $\noteqinpart{x_{\ordthree}}{\parti{A}_{\ordthree}}{y_{\ordthree}}$.
  It follows from this both that $\parti{Q} \notin \chain{A}$,
  and that $\parti{Q} \not ≤ \parti{A}_{\ordthree}$ for every $\parti{A}_{\ordthree}\in\chain{A}$.

  We now prove non-maximality of $\chain{A}$ by showing that the extension $\chain{A}\cup\{\parti{Q}\}$ is an antichain.
  Consider first the case when $\parti{A}_{\ordthree}$ is small. By construction of
  $S$, there is at least one non-singleton block of
  $\parti{A}_{\ordthree}$ that is contained in $S$ (as $S$ is
  precisely the union of such blocks for all the small partitions). On
  the other hand, each non-singleton block of $\parti{Q}$ contains
  exactly one $x_{\ordthree}$ and all its other elements are
  $y_{\ordthree}$, chosen from $\cardone\setsub S$. Hence, each block of $\parti{Q}$ 
  contains at most one element of $S$ and thus contains no
  non-singleton block of $\parti{A}_{\ordthree}$, whereby
  $\parti{A}_{\ordthree} \not ≤ \parti{Q}$.

  Second, consider the case when $\parti{A}_{\ordthree}$ is not small. By 
  construction, the union of all the non-singleton blocks of
  $\parti{Q}$ is the set
  $\setof{x_{\ordthree}}{\ordthree < \order{\cardtwo}} \cup
  \setof{y_{\ordthree}}{\ordthree < \order{\cardtwo}}$,
  of cardinality $≤ \cardtwo$, while the union of all
  non-singleton blocks of $\parti{A}_{\ordthree}$ has cardinality
  $> \cardtwo$ (as it is not a small partition). 
  Thus, it is not
  possible for the non-singleton blocks of $\parti{A}_{\ordthree}$
  to be contained in the non-singleton blocks of $\parti{Q}$, and we conclude
  $\parti{A}_{\ordthree} \not ≤ \parti{Q}$.
  Thus, $\parti{Q}$ is not in the antichain and is not
  comparable with its elements, hence $\chain{A}$ is not maximal.

  \smallskip

  \textbf{Finite antichains.} We finally turn to the case where
  $\cardtwo = n$ is \emph{finite}. The proof is nearly the same as in the infinite case. We
  now call a partition $\parti{A}_{\ordthree}$ \emph{small} if the union
  of its non-singleton blocks is finite.
  The proof in the case of no small $\parti{A}_{\ordthree}$ is the same
  with $\parti{P}_0 = \{0, x_1, …, x_n\}$ finite, hence
  $\parti{P}$ is small.
  For the case where $\chain{A}$ contains small partitions, since there are only
  finitely many of them, $Δ$ is finite, hence $S$ is also finite.
  Therefore, the union of all non-singleton blocks of $\parti{Q}$ is
  finite, and the rest of the proof holds.
  
  \smallskip

  In conclusion, any antichain of cardinality strictly less than $\cardone$ is
  not maximal, so any maximal antichain in $\thelattice{\cardone}$
  has cardinality at least $\cardone$.
\end{proof}

Note that, while we here restrict our attention to infinite partition lattices,
for finite $n$, maximal non-trivial antichains in
$\thelattice{n}$ can be proved to have cardinality at least $n$ by
induction on $n$.

\smallskip

Theorems \ref{thm:antichain_all_card} and \ref{thm:antichain_tight_bounds} 
tell us that any maximal antichain in $\thelattice{\cardone}$ has
cardinality between $\cardone$ and $2^{\cardone}$ and since both of
the bounds occur as cardinality of a maximal antichain, they are as
tight as possible. 
However, the following construction shows that, in some models where GCH is 
violated, there are maximal antichains of cardinalities between these bounds.
Specifically, it is true whenever there exists $\cardtwo < \cardone < 2^\cardtwo < 2^\cardone$.
% k < k+ < 2^k
\begin{remark}\label{rem:antichain_between_card}
  Let $\cardtwo < \cardone$ be two infinite cardinals. Given a maximal
  antichain $\chain{A}$ in $\thelattice{\cardtwo}$, we can construct
  a maximal antichain $\chain{B}$ of cardinality $\max(\cardone,2^\cardtwo)$
  in $\thelattice{\cardone}$ as follows:

  For each partition $\parti{A} \in \chain{A}$, let $\parti{A'}$ be a
  partition of $\cardone$ constructed from $\parti{A}$ by adding all
  the singleton blocks $\{x\}, x \in \cardone \setminus \cardtwo$ and
  let $\chain{A'}$ be the collection of all these. Next, let
  $\parti{P}_{\ordone, \ordtwo}$ be the singular partition of
  $\thelattice{\cardone}$ with non-singleton block
  $\{\ordone, \ordtwo\}$ and define
  \[
  \chain{B} = \chain{A'} \cup \setof{\parti{P}_{\ordone,
      \ordtwo}}{\ordone < \ordtwo < \order{\cardone} \text{ and } \cardtwo
    ≤ \card{\ordtwo}}
  \]

  \textbf{$\chain{B}$ is an antichain.} Indeed, the $\parti{A'}$ are
  mutually incomparable because $\chain{A}$ is an antichain; the
  $\parti{P}_{\ordone, \ordtwo}$ are mutually incomparable by
  construction; and a $\parti{A'}$ is not comparable with a
  $\parti{P}_{\ordone, \ordtwo}$ because each of the non-singleton
  blocks of the former is included in $\cardtwo$ while the
  non-singleton block of the latter contains an element of
  $\cardone \setminus \cardtwo$ (as we have only taken the $\ordtwo$
  with $\card{\ordtwo} ≥ \cardtwo$).

  \textbf{$\chain{B}$ is maximal.} Indeed, consider a partition
  $\parti{Q'}$ of $\thelattice{\cardone}$ which is not in
  $\chain{B}$. If the elements of $\cardone \setminus \cardtwo$ are
  all in singleton blocks in $\parti{Q'}$ then, by maximality of $\chain{A}$, its restriction
  $\parti{Q} = \restrict{\parti{Q}'}$ is comparable with some $\parti{A}$, 
  and so $\parti{Q'}$ is comparable with
  $\parti{A'}$. On the other hand, if there is a non singleton block
  containing $\ordone, \ordtwo$ with
  $\ordtwo \in \cardone \setminus \cardtwo$ we immediately have
  $\parti{P}_{\ordone, \ordtwo} ≤ \parti{Q}$.

  Because $\cardtwo < \cardone$, there are
  $\card{\cardone \setsub \cardtwo} = \cardone$ different $\ordtwo$ with
  $\cardtwo ≤ \card{\ordtwo}$, thus there are $\cardone$ different
  $\parti{P}_{\ordone, \ordtwo}$ in $\chain{B}$. Because the
  $\parti{A'}$ and $\parti{P}_{\ordone, \ordtwo}$ are distinct,
  $\chain{B}$ has cardinality
  $\cardinal{\chain{A'}} + \cardone = \max(\cardone,
  \cardinal{\chain{A}})$.
  Since we know, by Theorem~\ref{thm:antichain_all_card}, that we can
  build a maximal antichain of cardinality $2^{\cardtwo}$ in
  $\thelattice{\cardtwo}$, we can use this construction to build
  maximal antichains of cardinality $\max(\cardone, 2^{\cardtwo})$ in
  $\thelattice{\cardone}$, for any value of $\cardtwo <
  \cardone$.

  \smallskip

  Thus, for every cardinality $\cardtwo$ such that 
  $\cardtwo < \cardone < 2^\cardtwo$, we can build an antichain
  with cardinality $2^\cardtwo$ with $\cardone < 2^\cardtwo < 2^\cardone$.

  We do not yet know whether all cardinalities
  between the bounds always occur as antichains. 
  Certainly, we can construct non-GCH models
  where there exist cardinals between $\cardone$ and 
  $2^\cardone$ that cannot be written as $2^\cardtwo$. It is not yet known whether
  they can be realized as the cardinality of a maximal 
  antichain through another construction.
\end{remark}
% Do all the cardinalities between the bounds occur?

\section{Complements}
Recall that in a bounded lattice $L$, elements $a,b \in L$ are
\emph{complements} if and only if $a \lub b = \top$ and
$a \glb b = \bot$. We denote by $\compl{\parti{P}}$ the set of all
complements to $\parti{P}$ in $\thelattice{\cardone}$.  For finite
$\cardone = n$, counting the number of elements in $\compl{\parti{P}}$
is a difficult combinatorial problem. The best known estimate, due to
Grieser \cite{Grieser1991144}, is that if
$\parti{P} = \{B_1,\ldots,B_m\}$ is a partition in $\thelattice{n}$,
then the number of complements $\parti{Q}$ of $\parti{P}$ satisfying
$\cardinal{\parti{Q}} = n - m +1$ is
$\prod_{i=1}^m \cardinal{B_i} \cdot (n-m+1)^{m-2}$.

In the following, we prove a succession of lemmas leading up to the
main result of the section, Theorem \ref{thm:number_of_complements},
which gives a complete characterization of the counts of complements
to partitions of infinite cardinals.

For later use, we first recall some fundamental results on cardinal
arithmetic:
\begin{lemma}\label{lem:basic_arith}
  Let $(\cardone_i)_{i \in I}$ be a family of cardinals.  The
  following hold:
  \begin{enumerate}
  \item \label{lem:arith:sum} {\em (\cite{HolzSteffensWeitz}, Lemma
      1.6.3(b.i))} If $\cardone_i > 0$ for all $i \in I$, $I \neq
    \emptyset$ and at least one of the cardinals $\vert I \vert$ and
    $\cardone_i$ (for some $i \in I$) is infinite, then
    \[
    \sum_{i\in I} \cardone_i = \max\{\vert I
    \vert,\sup\setof{\cardone_i}{i \in I}\} = \vert I \vert \cdot
    \sup\setof{\cardone_i}{i \in I}\]
  \item \label{lem:arith:prod} {\em (\cite{HolzSteffensWeitz}, Lemma
      1.6.15(a), Tarski),} If $\cardtwo ≥ \aleph_0$ is a cardinal
    and
    $\sequenceof{\cardone_{\ordone}}{\ordone < \cardtwo}$ is an
    increasing sequence of infinite cardinals, then \[\prod_{\ordone <
      \cardtwo} \cardone_{\ordone} =
    (\sup\setof{\cardone_{\ordone}}{\ordone < \cardtwo})^\cardtwo\]
  \item \label{lem:arith:pow} {\em (\cite{HolzSteffensWeitz}, Lemma
      1.6.15(d), Tarski)} If $\cardone$ is an infinite cardinal,
    then
    \[2^{\cardone} = \left(\sup_{\cardtwo<\cardone}
      2^\cardtwo\right)^{\cofinality{\cardone}}\]
  \end{enumerate}
\end{lemma}

\begin{lemma}\label{lem:general_condition_two}
  Let $\cardone$ be an infinite cardinal, and $\parti{P}\notin \{\bot,\top\}$ be a
  partition of $\cardone$. Then, $\cardinal{\compl{\parti{P}}} ≥
  2^{\cardinal{\parti{P}}}$.
\end{lemma}

\begin{proof}
  If $\cardinal{\parti{P}}$ is finite, then there must exist at least
  one block $B$ of cardinality $\cardone$ in $\parti{P}$ (otherwise,
  since $\cup \parti{P} = \cardone$, we would have
  $\cofinality{\cardone}$ finite). Any singular partition whose
  non-singleton block contains exactly one element from each block of
  $\parti{P}$ is a complement to $\parti{P}$. Since there are
  $\cardone$ choices for the element in $B$, there are at least
  $\cardone$ complements to $\parti{P}$. Thus
  $\cardinal{\compl{\parti{P}}} ≥ \cardone >
  2^{\cardinal{\parti{P}}}$.

  Assume now that $\cardinal{\parti{P}}$ is infinite. Since $\parti{P}
  \ne \bot$, we can choose a block $B_0$ from $\parti{P}$ containing
  distinct elements $ι \ne υ$.
  Write $\parti{P}' = \parti{P}\setsub \{B_0\}$, and note that
  $\cardinal{\parti{P}'} = \cardinal{\parti{P}}$.  By the Axiom of
  Choice, we may select from each block $B\in \parti{P}'$ an element
  $\ordfour(B)$. Given any subset $\parti{P}_1 \subseteq \parti{P}'$,
  let $\parti{P}_2 = \parti{P}' \setsub \parti{P}_1$. If we now define
  \begin{equation*}
    \begin{split}
      Q_1 & = \{ι\} \cup \setof{\ordfour(B)}{B\in \parti{P}_1}\\
      Q_2 & = \{υ\} \cup \setof{\ordfour(B)}{B\in \parti{P}_2}\\
      \parti{Q}_s & = %'s'ingletons
      \setof{\{\ordfour\}}{\ordfour\in\cardone\setsub(Q_1\cup Q_2)}
    \end{split}
  \end{equation*}
  then $\parti{Q} = \{Q_1,Q_2\} \cup \parti{Q}_s$ is a
  complement to $\parti{P}$, as can be verified as follows:
  (i) $\parti{Q}$ is a partition, since each element is included in
  exactly one block. (ii) $\parti{P} \glb \parti{Q} = \bot$, since
  each block $Q\in\parti{Q}$ contains at most one element from each
  block in $\parti{P}$.  (iii) Finally,
  $\parti{P} \lub \parti{Q} = \top$, i.e.
  $\eqinpart{x}{\parti{P} \lub \parti{Q}}{y}$ for all
  $x, y \in \cardone$: Consider $x, y \in \cardone$, and write
  $B_x, B_y$ for the blocks in $\parti{P}$ containing $x$ and $y$,
  respectively. If $B_x, B_y \in \parti{P}_1 \cup \{B_0\}$, then there
  exists\footnotemark\ a $\ordfour_x \in Q_1 \cap B_x$ and a
  $\ordfour_y \in Q_1 \cap B_y$. This yields
  \[
  \bigeqinpart{x}{B_x}{%
    \bigeqinpart{\ordfour_x}{Q_1}{%
      \bigeqinpart{\ordfour_y}{B_y}{y}}}
  \]
  \footnotetext{Namely, $\ordfour_x = \ordfour(B_x)$ if
    $B_x \in \parti{P}_1$ and $\ordfour_x = ι$ if $B_x = B_0$.}

  If instead $B_x \in \parti{P}_1 \cup \{B_0\}$ and
  $B_y \in \parti{P}_2$, there is a $\ordfour_x \in Q_1 \cap B_x$, whereby
  % and $\ordfour_y \in Q_2 \cap B_y$
  \[
  \bigeqinpart{x}{B_x}{%
    \bigeqinpart{\ordfour_x}{Q_1}{%
      \bigeqinpart{ι}{B_0}{%
        \bigeqinpart{υ}{Q_2}{%
          \bigeqinpart{\ordfour(B_y)}{B_y}{y}}}}}
  \]
  The remaining two cases, $B_x,B_y \in \parti{P}_2 \cup \{B_0\}$ and
  $B_x \in \parti{P}_2 \cup \{B_0\}$, $B_y \in \parti{P}_1$, are
  symmetrical. Thus $\parti{Q} \in \compl{\parti{P}}$.  Clearly, two
  different choices of the subset $\parti{P}_1 \subseteq \parti{P}'$
  yields different complements $\parti{Q}$, whereby
  $\cardinal{\compl{\parti{P}}} ≥ 2^{\cardinal{\parti{P}}}$.
\end{proof}

\begin{lemma}\label{lem:noblock_high}
  Let $\cardone$ be an infinite cardinal, and $\parti{P}\notin \{\bot,\top\}$ be a
  partition of $\cardone$.  If $\parti{P}$ contains no block of
  cardinality $\cardone$, then
  $\cardinal{\compl{\parti{P}}} = 2^{\cardone}$
\end{lemma}

\begin{proof}
  If $\parti{P}$ contains no block of cardinality $\cardone$, then
  $\cardinal{\parti{P}} ≥ \cofinality{\cardone}$ (because
  $\cardone = \bigcup_{B\in \parti{P}} B$). Then
  Lemma~\ref{lem:general_condition_two} implies
  $\cardinal{\compl{\parti{P}}} ≥ 2^{\cardinal{\parti{P}}} ≥
  2^{\cofinality{\cardone}}$.
  Thus, if $\cardone$ is a regular cardinal, or if
  $\cardinal{\parti{P}} = \cardone$, we immediately obtain
  $\cardinal{\compl{\parti{P}}} = 2^\cardone$.

  Assume now that $\cardone$ is singular and
  $\cofinality{\cardone} ≤ \cardinal{\parti{P}} < \cardone$.  We can
  construct a complement $\parti{Q}$ to $\parti{P}$ as follows:
  \begin{enumerate}
  \item Let $A_0$ be a set containing exactly one element
    $\ordfour(B)$ from each block $B$ of $\parti{P}$. Let
    $\parti{P}' = \setof{B \setsub A_0}{B \in \parti{P}} \setminus \{
    \emptyset\}$, whereby $\cup\parti{P}' = \cardone\setsub A_0$.
  \item Any partition $\parti{Q}$ of $\cardone$ that contains $A_0$ as
    a block will have $\parti{P} \lub \parti{Q} = \top$: If
    $\ordtwo \in B_1$ and $\ordthree \in B_2$ for
    $B_1,B_2 \in \parti{P}$, then $\eqinpart{\ordtwo}{B_1}{%
      \eqinpart{\ordfour(B_1)}{A_0}{%
        \eqinpart{\ordfour(B_2)}{B_2}{\ordthree}}}$. %
    Hence, if a partition $\parti{Q}'$ of $\cardone\setsub A_0$
    satisfies $\cardinal{A'\cap B'} ≤ 1$ for all
    $A'\in \parti{Q}', B'\in\parti{P}'$, then
    $\parti{Q} = \{A_0\}\cup \parti{Q}'$ is a complement to
    $\parti{P}$.
  \item Because $\cardinal{A_0} = \cardinal{\parti{P}} < \cardone$, we
    have $\cardinal{\cardone \setsub A_0} = \cardone$, and hence
    $\sum_{B \in \parti{P}'} \cardinal{B} = \cardinal{\cup \parti{P}'}
    = \cardone$.
    By Lemma \ref{lem:basic_arith}(\ref{lem:arith:sum}), and using
    $\cardinal{\parti{P}'} < \cardone$, we obtain
    \[
    \cardone = \sum_{B \in \parti{P}'} \cardinal{B} = \max
    \left\{\cardinal{\parti{P}'}, \sup_{B \in \parti{P}'}
      \cardinal{B}\right\} = \sup_{B \in \parti{P}'} \cardinal{B}
    \]
  \item By definition of cofinality, there exists an increasing
    sequence
    $\sequenceof{\cardthree_\ordone}{\ordone < \cofinality{\cardone}}$
    of infinite cardinals strictly less than $\cardone$ that sums to
    $\cardone$. Because $\cardone$ is singular, Lemma
    \ref{lem:basic_arith}(\ref{lem:arith:sum}) implies
    $\cardone = \sup_{\ordone<\cofinality{\cardone}}
    \cardthree_\ordone$.

    Since also $\cardone = \sup\setof{|B|}{B \in \parti{P}'}$, we can
    choose by AC for every $\ordone<\cofinality{\cardone}$ some
    $B_\ordone\in \parti{P}'$ such that
    $|B_\ordone| ≥ \cardthree_\ordone$.  The sequence
    $\sequenceof{|B_\ordone|}{\ordone<\cofinality{\cardone}}$ clearly
    has supremum $\cardone$.

  \item For every successor ordinal
    $\ordone+1 < \cofinality{\cardone}$, split the block
    $B_{\ordone+1}$ into a small subset $B^-_{\ordone+1}$ of
    cardinality $\cardinal{B_\ordone}$ and a large subset
    $B^+_{\ordone+1}$ of cardinality $\cardinal{B_{\ordone+1}}$.  Then
    for every limit ordinal $\ordone$, define
    $B^+_\ordone = B_\ordone$.  Now choose for every ordinal
    $\ordone < \cofinality{\cardone}$ a bijection
    $σ_\ordone\colon B^+_\ordone \to B^-_{\ordone+1}$; let
    $A_{\ordone}^{\ordtwo} = \{ \ordtwo, σ_{\ordone}(\ordtwo) \}$ and
    let
    $A_\ordone = \setof{ A_{\ordone}^{\ordtwo}}{\ordtwo\in
      B^+_\ordone}$,
    it is a partition of $B_{\ordone}^+ \cup B_{\ordone+1}^-$,
    consisting of two-element blocks.
  \item Finally,
    % fill up everything that we haven't touched with singletons:
    let
    \[ S = \cardone\setsub \left(A_0\cup
      \bigcup_{\ordone<\cofinality{\cardone}} B_\ordone\right).
    \]
    Then it is easy to verify that
    \[\parti{Q} = \setof{\{\ordtwo\}}{\ordtwo\in S} \cup \{A_0\}
    \cup \bigcup_{\ordone < \cofinality{\cardone}} A_\ordone
    \]
    is a complement to $\parti{P}$. Indeed, each element of
    $A_{\ordone}$ contains exactly two elements from two different
    blocks $B_{\ordone}$ and $B_{\ordone+1}$ of $\parti{P}'$, and the
    other blocks are either $A_0$ or singletons. Thus, for any
    $A' \in \parti{Q}$ other than $A_0$ and any $B' \in \parti{P}'$,
    we have $\cardinal{A' \cap B'} ≤ 1$, and by Point (2) above,
    $\parti{Q}$ is a complement to $\parti{P}$.
  \end{enumerate}
  For $σ_{\ordone} \neq σ'_{\ordone}$ there exists
  $\ordtwo \in B^+_\ordone$ with
  $σ_{\ordone}(\ordtwo) \neq σ'_{\ordone}(\ordtwo)$ and hence
  $\{\ordtwo,σ_{\ordone}(\ordtwo)\} \neq
  \{\ordtwo,σ'_{\ordone}(\ordtwo)\}$,
  whence each choice of
  $\sequenceof{σ_\ordone}{\ordone<\cofinality{\cardone}}$ yields a
  distinct complement.  There are
  $\cardinal{B^-_{\ordone+1}}^{\cardinal{B^+_\ordone}} =
  \cardinal{B_\ordone}^{\cardinal{B_\ordone}} =
  2^{\cardinal{B_\ordone}}$
  ways to choose each bijection $σ_\ordone$. Since
  $\cardone = \sup_{\ordone < \cofinality{\cardone}}
  \cardinal{B_{\ordone}}$,
  for each $\cardtwo < \cardone$, there exists $\ordone$ with
  $\cardinal{B_{\ordone}} > \cardtwo$ hence
  $2^{\cardinal{B_{\ordone}}} ≥ 2^{\cardtwo}$. Consequently,
  $\sup_{\ordone < \cofinality{\cardone}} 2^{\cardinal{B_{\ordone}}} ≥
  \sup_{\cardtwo < \cardone} 2^{\cardtwo}$.
  Then, Lemma \ref{lem:basic_arith}(\ref{lem:arith:prod}) and
  (\ref{lem:arith:pow}) yields
  \[
  \cardinal{\compl{\parti{P}}} ≥ \prod_{\ordone<\cofinality{\cardone}}
  2^{\cardinal{B_\ordone}} =
  \left(\sup_{\ordone<\cofinality{\cardone}} 2^{\cardinal{B_\ordone}}
  \right)^{\cofinality{\cardone}} ≥ \left(\sup_{\cardtwo<\cardone}
    2^{\cardtwo} \right)^{\cofinality{\cardone}} = 2^\cardone
  \]
  As $2^{\cardone} = \cardinal{\thelattice{\cardone}}$ is an upper
  bound to the number of complements to $\parti{P}$, we obtain
  $\cardinal{\compl{\parti{P}}} = 2^{\cardone}$.
\end{proof}

\begin{lemma}\label{lem:oneblock_high}
  Let $\cardone$ be an infinite cardinal, and $\parti{P} \in \thelattice{\cardone}$
  be any partition of $\cardone$.  If $\parti{P}$ contains a block
  $B$ of cardinality $\cardone$, then
  $\cardinal{\compl{\parti{P}}} = \cardone^{\cardtwo}$, where
  $\cardtwo = \cardinal{\cardone \setsub B}$.
\end{lemma}

\begin{proof}
  Assume that $\parti{P}$ has a block $B$ of cardinality $\cardone$,
  and write $\parti{P}$ as the disjoint union
  $\parti{P} = \{B\}\cup\parti{P}'$.  Let
  $\bar{B} = \cardone \setminus B$, and observe that
  $\bar{B} = \cup\parti{P}'$.  Denote
  $\cardtwo = \cardinal{\cup\parti{P}'} = \cardinal{\bar{B}} =
  \cardinal{\cardone \setsub B} ≤ \cardone$.
  If $\cardinal{\bar{B}} = 0$, then $\parti{P} = \top$, which has
  exactly one complement, namely $\bot$, whereby
  $\cardinal{\compl{\parti{P}}} = 1 = \cardone^0$ as desired.  Hence,
  in the following we can assume that $\parti{P} \ne \top$, such that
  $1 ≤ \cardtwo ≤ \cardone$.

  \textbf{Lower bound.} To show
  $\cardinal{\compl{\parti{P}}} ≥ \cardone^\cardtwo$, choose any
  injection $σ \colon \bar{B} \to B$, and write
  $B_\cardtwo = σ \left(\bar{B} \right)$,
  $B_s = B \setsub B_\cardtwo$. % 's' for 'singletons'.
  If we now define $A_\ordone = \{ \ordone, σ(\ordone) \}$, then the
  partition
  $\parti{Q} = \setof{ A_\ordone }{ \ordone \in \bar{B}} \cup
  \setof{\{\ordtwo\}}{\ordtwo \in B_s}$
  is a complement to $\parti{P}$, as is easily verified by checking
  each of the properties: (i) Every element of $\cardone$ is either in
  $\bar{B}$, $B_\cardtwo$, or $B_s$, hence $\parti{Q}$ is a partition;
  (ii) each block of $\parti{Q}$ is either a singleton (from $B_s$) or
  a doubleton (one of the $A_{\ordone}$) with one element in $B$ and
  one out, hence it intersects each block of $\parti{P}$ in at most
  one point; (iii) any $\ordone$ and $\ordone'$ are linked through
  their images $σ(\ordone)$ and $σ(\ordone')$ in $B$.

  There are $\cardone^\cardtwo$ ways of choosing $σ$, and each way
  leads to a distinct complement.  Hence,
  $\cardinal{\compl{\parti{P}}} ≥ \cardone^\cardtwo$.

  \textbf{Upper bound.} Observe that we have
  $\cardinal{\bar{B}} ≥
  \cardinal{\parti{P}'}$.
  Let now $\parti{Q}$ be any complement to $\parti{P}$. Then
  \begin{enumerate}
  \item $\parti{P} \glb \parti{Q} = \bot$ implies that for every $A
    \in \parti{Q}$ and $B \in \parti{P}$, $\cardinal{A \cap B} ≤ 1$,
    and hence $\cardinal{A} ≤ \cardinal{\parti{P}} ≤ \cardtwo$.
  \item Let $\parti{Q}' = \setof{A' \in \parti{Q}}{\cardinal{A'} ≥ 2}$
    be the set of non-singleton blocks of $\parti{Q}$. Since
    $\cardinal{A' \cap B} ≤ 1$, each $A'$ must intersect $\bar{B}$ in
    at least one point. As blocks of a partition, the $A'$ are
    pairwise disjoint, as are consequently these intersections. Hence
    choosing (by AC) an element in each $A' \cap \bar{B}$ defines an
    injection from $\parti{Q}'$ to $\bar{B}$, whereby
    $\cardinal{\parti{Q}'} ≤ \bar{B} ≤ \cardtwo$.
  \item An upper bound for the number of complements to $\parti{P}$
    can then be found in the following way: Specifying $\parti{Q}'$
    uniquely determines the complement $\parti{Q}$. $\parti{Q}$ has
    $\cardinal{\parti{Q}'} ≤ \cardtwo$ non-singleton blocks, each of
    size at most $\cardinal{\parti{P}} ≤ \cardtwo$, yielding
    $\cardinal{\bigcup \parti{Q}'} ≤ \cardinal{\cardtwo \times
      \cardtwo} = \cardtwo$.
    A complement $\parti{Q}$ is fully specified by (i) the union
    $\bigcup\parti{Q}'$ of the non-singleton blocks; and (ii) the
    partition of these unions into blocks of $\parti{Q}'$.  Write
    $ε = \cardinal{\bigcup \parti{Q}'}$. For each cardinality $ε$ that
    $\bigcup\parti{Q}'$ can attain, there are at most
    $\cardtwo^{ε} ≤ \cardone^ε$ ways to select the elements
    $\bigcup\parti{Q}'$ from $\bar{B}$. Since $\parti{Q}'$ is a
    partition of $ε$, there are at most
    $\cardinal{\thelattice{ε}} = 2^ε$ distinct ways to partition these
    elements into blocks of $\parti{Q}'$.
    Letting now $ε$ range over all potentially allowed cardinalities,
    i.e. all less than or equal to $\cardtwo$, we find
    \[
    \cardinal{\compl{\parti{P}}} ≤ \sum_{ε ≤ \cardtwo} 2^ε \cdot
    \cardtwo^{ε} ≤ \sum_{ε ≤ \cardtwo} 2^ε \cdot
      \cardone^{ε} = \sum_{ε ≤ \cardtwo} \cardone^{ε} =
    \cardone^\cardtwo
    \]
  \end{enumerate}
  The above applies both to finite and infinite
  $\cardtwo = \cardinal{\cardone \setsub B}$.
\end{proof}

\begin{corollary}\label{cor:compl_largeblocks}
  Let $\cardone$ be infinite and $\parti{P}\notin \{\bot,\top\}$ be a
  partition of $\cardone$.  If $\parti{P}$ contains two or more blocks
  of cardinality $\cardone$, then
  $\cardinal{\compl{\parti{P}}} = 2^\cardone$.
\end{corollary}

\begin{proof}
  If two blocks, $B_1$ and $B_2$, have cardinality
  $\cardone$, then $\cardinal{\cardone\setsub B_1} = \cardone$, and
  Lemma \ref{lem:oneblock_high} yields
  $\cardinal{\compl{\parti{P}}} ≥ \cardone^\cardone = 2^\cardone$.
\end{proof}

\begin{lemma}\label{lem:general_condition_one}
  Let $\cardone$ be infinite and $\parti{P} \notin \{\bot,\top\}$ be a
  partition of $\cardone$.  Then,
  $\cardone ≤ \cardinal{\compl{\parti{P}}} ≤ 2^\cardone$.
\end{lemma}

\begin{proof}
  As $\cardinal{\thelattice{\cardone}} = 2^{\cardone}$, the upper
  bound is immediate, and it suffices to prove the lower bound.  There
  are two cases to consider: First, if $\parti{P}$ contains no block
  of cardinality $\cardone$, Lemma \ref{lem:noblock_high} yields
  $\cardinal{\compl{\parti{P}}} = 2^\cardone $. Otherwise, if
  $\parti{P}$ does contain a block $B$ of cardinality
  $\cardone$, Lemma \ref{lem:oneblock_high} implies
  $\cardinal{\compl{\parti{P}}} =
  \cardone^{\cardinal{\cardone\setsub B}}$.
  Because $\parti{P}\ne\top$, we have
  $\cardinal{\cardone\setsub B}≥ 1$, whereby
  $\cardinal{\compl{\parti{P}}} ≥ \cardone$.  In both cases,
  $\cardinal{\compl{\parti{P}}} ≥ \cardone$.
\end{proof}
%\clearpage
We can now state our main theorem:

\begin{theorem}\label{thm:number_of_complements}
  Let $\cardone$ be infinite and $\parti{P}\notin \{\bot,\top\}$
  be a partition in $\thelattice{\cardone}$. Then
  \begin{enumerate}
  \item $\cardone ≤ \cardinal{\compl{\parti{P}}} ≤
    2^{\cardone}$.
  \item $\cardinal{\compl{\parti{P}}} ≥ 2^{\cardinal{\parti{P}}}$.
  \item If $\parti{P}$ contains no block of cardinality $\cardone$, then
    $\cardinal{\compl{\parti{P}}} = 2^{\cardone}$.
  \item If $\parti{P}$ contains a block $B$ of cardinality
    $\cardone$, then
    $\cardinal{\compl{\parti{P}}} =
    \cardone^{\cardinal{\cardone\setsub B}}$.
  \item If $\parti{P}$ contains two or more blocks of cardinality
    $\cardone$, then $\cardinal{\compl{\parti{P}}} = 2^\cardone$.
  \end{enumerate}
\end{theorem}

\begin{proof}
  (1)--(4) are Lemmas \ref{lem:general_condition_one},
  \ref{lem:general_condition_two}, \ref{lem:noblock_high} and
  \ref{lem:oneblock_high}, respectively.  (5) is Corollary
  \ref{cor:compl_largeblocks}.
\end{proof}

A consequence of Theorem \ref{thm:number_of_complements} is that
partitions with fewer complements than $2^{\cardone}$ must always have
exactly one large block, and a sufficiently small number of elements
remaining after removing it:
\begin{corollary}
  If $\cardinal{\compl{\parti{P}}} < 2^\cardone$, then $\parti{P}$
  contains exactly one block $B$ of size $\cardone$, and
  $\cardinal{\cardone\setsub B}<\cardone$.
\end{corollary}

For any infinite cardinal $\cardone$, both $\cardone$ and
$2^{\cardone}$ can be realized as $\cardinal{\compl{\parti{P}} }$ for
some partition $\parti{P}$.  In fact, Theorem
\ref{thm:number_of_complements} provides a complete characterization
of the cardinals that can be realized as complement counts: it is
precisely those cardinals of the form $\cardone^\cardtwo$, with
$0≤ \cardtwo≤ \cardone$ (and $1≤ \cardtwo ≤ \cardone$ when considering
only non-trivial partitions), as the following corollary shows:

\begin{corollary}\label{cor:compl-existence}
  For any infinite cardinality $\cardone$, and every cardinal $0 ≤
  \cardtwo ≤ \cardone$, there is a partition $\parti{P}_\cardtwo$ of
  $\cardone$ for which $\cardinal{\compl{\parti{P}_\cardtwo}} =
  \cardone^\cardtwo$.

  In particular, there exist partitions $\parti{P}$ and $\parti{R}$ in
  $\thelattice{\cardone}$ with $\cardinal{\compl{\parti{P}}} =
  \cardone$, respectively $\cardinal{\compl{\parti{R}}} =
  2^{\cardone}$.

  No cardinal that is not of the form $\cardone^\cardtwo$ can be
  realised as $\cardinal{\compl{\parti{P}}}$.
\end{corollary}

\begin{proof}
  Theorem~\ref{thm:number_of_complements} fully describes the possible
  number of complements of non-trivial partition. Remembering that
  $\cardone^{\cardone} = 2^{\cardone}$ and that, for $\top$ and
  $\bot$, the number of complements is $1 = \cardone^0$, it is
  apparent that for any partition, the number of complements has the
  form $\cardone^{\cardtwo}$. Conversely, the cardinal
  $\cardone^\cardtwo$, $1≤ \cardtwo ≤ \cardone$, can be realised
  as the number of complements to the two-block partition
  $\{B, \bar{B}\}$ with $\cardinal{B} = \cardtwo$ and
  $\cardinal{\bar{B}} = \cardone$, using
  Theorem~\ref{thm:number_of_complements}(4).
\end{proof}

\subsection{Orthocomplements}
It is well-known that $\thelattice{\cardone}$ is relatively
complemented. To our knowledge, it has so far been unknown whether
there exists an $n > 2$ such that $\thelattice{n}$ is
orthocomplemented. We prove that for any cardinality $\cardone > 2$
(finite or transfinite), $\thelattice{\cardone}$ is \emph{not}
orthocomplemented.

\begin{proposition}[Orthocomplements]
  If $\cardone > 2$, then $\thelattice{\cardone}$ is not
  orthocomplemented.
\end{proposition}

\begin{proof}
  Orthocomplementation yields a bijection between atoms and
  co-atoms. Atoms are singular partitions whose non-singleton block is
  a pair, hence there are as many atoms as there are pairs, namely
  $\frac{\cardone \times (\cardone-1)}{2}$. Co-atoms are partitions
  with two blocks, a non-trivial subset of $\cardone$ and its
  complement, hence there are as many co-atoms as half the number of
  non-trivial subsets, namely $\frac{2^{\cardone}-2}{2}$. For
  $\cardone ≥ 3$, these numbers are different hence
  $\thelattice{\cardone}$ cannot be orthocomplemented. For
  $\cardone = 3$, one can easily check that $\thelattice{3}$ is not
  orthocomplemented either (because it contains an odd number of
  elements).
\end{proof}

\section{Results under GCH}
\label{sec:GCHresults}

Under the Generalized Continuum Hypothesis, the results from
the previous sections all simplify greatly, and it is possible
to obtain much stronger results.

\subsection{Maximal chains under GCH}
Under GCH, we can fully determine the possible cardinals for maximal chains in
$\thelattice{\cardone}$, as we will see in Theorem \ref{thm:GCHmaxchain} below.

\newcommand{\Pminab}{\parti{P}^-_{α, β}}

\begin{lemma} \label{lem:maxchainUB} Let $\cardone$ be an infinite
  cardinal, let $\chain{C}$ be a maximal chain in
  $\thelattice{\cardone}$. Given a partition $\parti{P} \in \chain{C}$,
  write $\parti{P}^+ = \glb\upperchain C {\parti{P}}$
  (cf. Definition~\ref{def:upper-lower}), and let
  \[
    \chain{C}^* =  \setof{\parti{P}\in\chain{C}}{ \parti{P} \prec \parti{P}^+}
  \]
  Then $\cardone ≤ 2^{|\chain{C}^*|}$.
\end{lemma}

\begin{proof}
  First note that, being maximal, $\chain{C}$ is end-point including,
  closed and covering by
  Lemma~\ref{lem:closed-covering-saturated}. Let
  $\Pminab = \lub \lowerchain{C}{\ordone, \ordtwo}$ (cf.
  Definition~\ref{def:xy-upper-lower}). By
  Lemma~\ref{lem:meet-join-chain}, we have
  $\Pminab \prec \glb \upperchain{C}{\ordone, \ordtwo}$ and by
  construction,
  $\upperchain{C}{\ordone, \ordtwo} = \upperchain{C}{\Pminab}$. Hence,
  $\Pminab \in \chain{C}^*$.

  We next construct a family of maps
  $\varphi_\ordone\colon \chain{C}^* \to \{0,1\}$ in the following
  way: Let $\parti{P} \in \chain{C}^*$. By definition, we have
  $\parti{P} \prec \parti{P}^+$, thus there exist an unique block
  $B_{\parti{P}}$ of $\parti{P}^+$ which is the union of two blocks of
  $\parti{P}$ (and all others are also blocks of $\parti{P}$). Choose,
  by axiom of choice, one of these two blocks as $B_{\parti{P},0}$.

  Now, for each $α \in \cardone$, we define
  $φ_α\colon \chain{C}^* \to \{0, 1\}$ by
  \[
  φ_α(\parti{P}) =
    \begin{cases}
      0 & \text{ if $α\in B_{\parti{P}, 0}$ }\\
      1 & \text{ otherwise } %($α$ in the other block or not in $B_{\parti{P}}$)}
    \end{cases}
  \]

  The previous construction shows that for any $α \neq β$, (i)
  $\Pminab \in \chain{C}^*$; (ii) $α$ and $β$ are precisely in
  $B_{\Pminab}$; and (iii) they are in different blocks of $\Pminab$,
  that is one of them is in $B_{\Pminab, 0}$ and the other is
  not. Hence, $φ_α(\Pminab) \neq φ_β(\Pminab)$ and consequently,
  $φ_α \neq φ_β$. Thus, the map $\ordone\mapsto \varphi_\ordone$ is
  injective from $\cardone$ to $\{0, 1\}^{\chain{C}^*}$, whereby
  $\cardone ≤ 2^{\cardinal{\chain{C}^*}}$.
\end{proof}

\begin{proposition}
  \label{prop:GCHchainLB}
  Let $\cardone$ be any infinite cardinal. Under GCH, any maximal
  chain $\chain{C}$ in $\thelattice{\cardone}$ has
  cardinality
  \[
  \card{\chain{C}} ≥
  \begin{cases}
    \cardone^- & \text{if $\cardone$ is a successor cardinal},\\
    \cardone   & \text{if $\cardone$ is a limit cardinal.}
  \end{cases}
  \]
\end{proposition}

\begin{proof}
  By Lemma \ref{lem:maxchainUB},
  $\cardone ≤ 2^{\cardinal{\chain{C}^*}} ≤ 2^{\cardinal{\chain{C}}}$
  for any maximal chain $\chain{C}$. Under GCH,
  $\cardinal{\chain{C}}^+ = 2^{\cardinal{\chain{C}}}$, whereby
  $\cardone ≤ \cardinal{\chain{C}}^+$.  Assume now that
  $\cardinal{\chain{C}} < \cardone$, i.e.~
  $\cardinal{\chain{C}} < \cardone ≤ \cardinal{\chain{C}}^+$.  By
  definition of the successor relation, this implies
  $\cardone = \cardinal{\chain{C}}^+$.

  If $\cardone$ is a limit cardinal, this is a contradiction, and so
  we must have $\cardinal{\chain{C}} ≥ \cardone$. If $\cardone$ is a
  successor cardinal, it follows that
  $\cardinal{\chain{C}} ≥ \cardone^-$.
\end{proof}

By combining the previous results, we can tightly bound the cardinal
of any maximal chain in $\thelattice{\cardone}$ under GCH.  In
addition, the following simple restriction of Theorems
\ref{thm:long-nochainreach} and \ref{thm:small_chain} to the case when
GCH is assumed, provides instances of long and short maximal chains
that realize the bounds established above.
\begin{corollary}
  \label{cor:GCH-long-and-short-chains}
  Let $\cardone$ be a infinite cardinal. Under GCH, there exists a
  maximal chain of length $\cardone^+$ in $\thelattice{\cardone}$; and
  there exists a chain of length $\cardone$ in
  $\thelattice{\cardone^+}$.
\end{corollary}

\begin{proof}
  If GCH is assumed, the first statement is an immediate consequence
  of Theorem \ref{thm:long-nochainreach}; and the second statement
  follows, because the precondition in Theorem \ref{thm:small_chain},
  $\cardtwo < \cardone$ only if $2^\cardtwo < 2^\cardone$, is always
  satisfied under GCH.
\end{proof}

\begin{theorem} \label{thm:GCHmaxchain} Let $\cardone$ be an infinite
  cardinal. Under GCH, the cardinality of any maximal chain in
  $\thelattice{\cardone}$ is:
    \begin{itemize}
    \item $\cardone^-$, $\cardone$, or $\cardone^+$ (and all three are
      always achieved) if $\cardone$ is a successor; and
    \item either $\cardone$ or $\cardone^+$ (and both are achieved) if
      $\cardone$ is a limit cardinal.
    \end{itemize}
\end{theorem}

\begin{proof}
  The lower bounds are given by Proposition \ref{prop:GCHchainLB}, and
  the upper bound is the cardinality of $\thelattice{\cardone}$ itself.

  Furthermore, each possible value is always realized: For any
  infinite cardinal $\cardone$, there always exists a well-founded
  maximal chain of cardinality $\cardone$. By Corollary
  \ref{cor:GCH-long-and-short-chains}, there also exists a chain of
  length $\cardone^+$, and for successor cardinals $\cardone$, a chain
  of length $\cardone^-$.
\end{proof}

The bounds can also be stated in a symmetrical fashion:
\begin{corollary}
  For any infinite cardinal $\cardone$ (successor or limit), the
  cardinality of any maximal chain lies between
  $\sup\{\cardtwo < \cardone\}$ and $\inf \{\cardtwo > \cardone\}$.
\end{corollary}

\subsection{Maximal antichains under GCH}

\begin{corollary}
  Under GCH, when $\cardone$ is an infinite cardinal, the length of
  any maximal antichain in $\thelattice{\cardone}$ is either
  $\cardone$ or $\cardone^+$, and both are realized.
\end{corollary}

\begin{proof}
  This is the content of Theorems \ref{thm:antichain_all_card},
  \ref{thm:antichain_tight_bounds}, together with the fact that
  $\cardinal{\thelattice{\cardone}} = 2^\cardone$, when the assumption
  $2^\cardone = \cardone^+$ is made.
\end{proof}

\subsection{Complements under GCH}

Under GCH, the number of complements $\cardinal{\compl{\parti{P}}}$
for partitions $\parti{P}\notin\{\top,\bot\}$ of an infinite cardinal
$\cardone$ is either $\cardone$ or $\cardone^+ = 2^\cardone$. In
addition, the simplified rules for arithmetic under GCH strengthen
Theorem \ref{thm:number_of_complements}:

\begin{theorem}
  \label{thm:GCHkappa_complements}
  Let $\cardone$ be an infinite cardinal, and
  $\parti{P} \notin \{\top,\bot\}$ be a partition of $\cardone$.
  Assuming GCH, then
  \[
  \cardinal{\compl{\parti{P}}} =
  \begin{cases}
    \cardone & \text{if and only if exactly one block $B
      \in \parti{P}$ has $\cardinal{B} = \cardone$,}\\
    & \text{and $\cardinal{\cardone \setsub B} <
      \cofinality{\cardone}$}\\
    2^\cardone & \text{otherwise}
  \end{cases}
  \]
\end{theorem}

\begin{proof}
  Consider the three cases:
  \begin{enumerate}
  \item First, if $\parti{P}$ contains either zero or at least two
    blocks of size $\cardone$, then Theorem
    \ref{thm:number_of_complements}(2) or
    \ref{thm:number_of_complements}(5) yields
    $\cardinal{\compl{\parti{P}}} = 2^\cardone$.
  \item Next, assume $\parti{P}$ contains exactly one block $B$ of
    size $\cardone$, and
    $\cardinal{\cardone\setsub B} ≥ \cofinality{\cardone}$. By Theorem
    \ref{thm:number_of_complements}(4),
    $\cardinal{\compl{\parti{P}}} = \cardone^{\vert \cardone \setsub B
      \vert} ≥ \cardone^{\cofinality{\cardone}} > \cardone$.
    GCH, together with $\cardinal{\compl{\parti{P}}} ≤ 2^\cardone$,
    then yields $\cardinal{\compl{\parti{P}}} = 2^\cardone$.
  \item Finally, if $\parti{P}$ contains exactly one block $B$
    of size $\cardone$ in $\parti{P}$, and
    $\cardinal{\cardone\setsub B} < \cofinality{\cardone}$. Under
    GCH, $\cardone^\cardtwo = \cardone$ if and only if
    $1 ≤ \cardtwo < \cofinality{\cardone}$.  Together with Theorem
    \ref{thm:number_of_complements}(4) and $\parti{P}\ne\top$, this
    yields $\cardinal{\compl{\parti{P}}} =
    \cardone^{\cardinal{\cardone\setsub B}} = \cardone$.\\[-2em]
  \end{enumerate}
\end{proof}

In Theorem \ref{thm:GCHkappa_complements}, the assumption of GCH is
necessary in the sense that it is easy to construct non-GCH examples
that violate the result.  For example, consider a model where
$2^{\aleph_1} = 2^{\aleph_2} = \aleph_3$, which is consistent with ZFC
by Easton's Theorem, and let $\cardone = \aleph_2$. Consider a
partition $\parti{P} = \{B,B'\}$, where
$\cardinal{B} = \aleph_2$, and
$\cardinal{B'} = \aleph_1 < \cofinality{\aleph_2}$. Then
$\cardinal{\cardone\setsub B} = \cardinal{B'} = \aleph_1$, so
\ref{thm:number_of_complements}(4) yields
$\cardinal{\compl{\parti{P}}} = \aleph_2^{\aleph_1} ≥
2^{\aleph_1}$.
But in this model,
$2^{\aleph_1} = 2^{\aleph_2} = 2^\cardone > \cardone$, despite
$\cardinal{\cardone\setsub B} < \cofinality{\cardone}$.

However, the result does hold for some classes of cardinals,
regardless of whether GCH is assumed.  In particular, whenever
$\cardthree^\cardtwo < \cardone$ for all $\cardthree < \cardone$ and
all $\cardtwo < \cofinality{\cardone}$.  This is for example the case
when $\cardone$ is a strong limit cardinal.

\subsection{Gathering all results under GCH}

Collecting the results of this section gives us the
following concise characterization of chains, antichains,
and complements in infinite partition lattices under GCH:
\begin{theorem}
  \label{cor:allGCH}
  Under GCH, when $\cardone$ is an infinite cardinal:
  \begin{enumerate}
  \item Any maximal well-founded chain in $\thelattice{\cardone}$
    always has cardinality $\cardone$.
  \item Any general maximal chain in $\thelattice{\cardone}$ has
    cardinality
    \begin{enumerate}
    \item $\cardone^-$, $\cardone$, or $\cardone^+$ (and all three are
      always achieved) if $\cardone$ is a successor cardinal; and
    \item either $\cardone$ or $\cardone^+$ (and both are achieved) if
      $\cardone$ is a limit cardinal.
    \end{enumerate}
  \item Any non-trivial maximal antichain in $\thelattice{\cardone}$
    has cardinality either $\cardone$ or $\cardone^+$, and both are
    achieved.
  \item Any non-trivial partition has either $\cardone$ or
    $\cardone^+$ complements.  $\parti{P}\notin\{\bot,\top\}$ has
    $\cardone$ complements if and only if (i) $\parti{P}$ contains
    exactly one block, $B$, of cardinality $\cardone$, and (ii)
    $\cardinal{\cardone\setsub B} < \cofinality{\cardone}$; otherwise,
    $\parti{P}$ has $\cardone^+$ complements.
  \end{enumerate}
\end{theorem}

\section{Acknowledgments}
The authors thank Beno\^{\i}t Kloeckner for his many comments on a
preliminary version of this article, and thank the anonymous
referees for valuable suggestions.
James Avery was supported by VILLUM FONDEN through the network for
Experimental Mathematics in Number Theory, Operator Algebras, and
Topology.
Jean-Yves Moyen was partially supported by the ANR project ``Elica''
ANR-14-CE25-0005 and by the Marie Sk\l{}odowska--Curie action
  ``Walgo'', program H2020-MSCA-IF-2014, number 655222.
Pavel Ruzicka was partially supported by the Grant Agency of the Czech
Republic under the grant no. GACR 14-15479S.
Jakob Grue Simonsen was partially supported by the Danish Council for
Independent Research \emph{Sapere Aude} grant ``Complexity through Logic
and Algebra'' (COLA).
\bibliographystyle{alpha}
\bibliography{bibmacros,bibjym,bibliographie}

\end{document}